\documentclass[12pt]{amsart}
\usepackage{latexsym,amsmath,amssymb}
\usepackage{accents}
  
\title[Integro-differential harmonic maps]{\protect{Integro-differential harmonic maps into spheres}}

\author{Armin Schikorra}

\address{Armin Schikorra, Max-Planck Institut MiS Leipzig, Inselstr. 22, 04103 Leipzig, Germany, {\tt armin.schikorra@mis.mpg.de}}

\thanks{The research leading to these results has received funding from the European Research Council under the European Union's Seventh Framework Programme (FP7/2007-2013) / ERC grant agreement n° 267087.}

plus 6pt minus 12pt
\def\eps{\varepsilon}

\def\N{{\mathbb N}}

\def\S{{\mathbb S}}

\newtheorem{theorem}{Theorem}
\newtheorem{lemma}[theorem]{Lemma}

\newtheorem{proposition}[theorem]{Proposition}

\theoremstyle{definition}
\newtheorem{remark}[theorem]{Remark}
\newtheorem{definition}[theorem]{Definition}

\def\supp{{\rm supp\,}}

\newcommand{\R}{\mathbb{R}}

\newcommand{\Z}{\mathbb{Z}}
\newcommand{\vrac}[1]{\| #1 \|}

\newcommand{\brac}[1]{\left (#1 \right )}
\newcommand{\abs}[1]{\left |#1 \right |}

\newcommand{\Sw}{\mathcal{S}}

\newcommand{\barint}{
\rule[.036in]{.12in}{.009in}\kern-.16in \displaystyle\int }

\newcommand{\barcal}{\mbox{$ \rule[.036in]{.11in}{.007in}\kern-.128in\int $}}

\def\mvint_#1{\mathchoice
          {\mathop{\vrule width 6pt height 3 pt depth -2.5pt
                  \kern -8pt \intop}\nolimits_{\kern -3pt #1}}%
          {\mathop{\vrule width 5pt height 3 pt depth -2.6pt
                  \kern -6pt \intop}\nolimits_{#1}}%
          {\mathop{\vrule width 5pt height 3 pt depth -2.6pt
                  \kern -6pt \intop}\nolimits_{#1}}%
          {\mathop{\vrule width 5pt height 3 pt depth -2.6pt
                  \kern -6pt \intop}\nolimits_{#1}}}

\numberwithin{theorem}{section} \numberwithin{equation}{section}

\newcommand{\lap}{\Delta }
\newcommand{\aleq}{\precsim}
\newcommand{\ageq}{\succsim}
\newcommand{\aeq}{\approx}
\newcommand{\Rz}{\mathcal{R}}
\newcommand{\laps}[1]{\lap^{\frac{#1}{2}}}
\newcommand{\lapms}[1]{I^{#1}}

\newcommand{\cut}{\accentset{\bullet}{\chi}}
\newcommand{\cutA}{\accentset{\circ}{\chi}}
\newcommand{\scut}{\accentset{\bullet}{\eta}}
\newcommand{\scutA}{\accentset{\circ}{\eta}}

\begin{document}

\sloppy

\sloppy


\begin{abstract}
For $s \in (0,1)$ we introduce (integro-differential) harmonic maps $v: \Omega \subset \R^n \to \R^N$, which are defined as critical points of the Besov-Slobodeckij energy
\[
 \int\limits_{\Omega}\int\limits_{\Omega} \frac{|v(x)-v(y)|^{p_s}}{|x-y|^{n+sp_s}}\ dx\ dy,
\]
with the side-condition that $v(\Omega) \subset \S^{N-1}$, for the $(N-1)$-sphere $\S^{N-1} \subset \R^N$. If $p_s = 2$ this are the classical fractional harmonic maps first considered by Da Lio and Rivi\`{e}re. For $p_s \neq 2$ this is a new energy which has degenerate, non-local Euler-Lagrange equations. They are different from the $n$/$s$-harmonic maps introduced by Francesca Da Lio and the author, and have to be treated with new arguments, which might be of independent interest for further applications on geometric energies. For the critical case $p_s = \frac{n}{s}$ we show H\"older continuity of these maps.
\end{abstract}

\maketitle

\section{Introduction}
Let $\Omega \subset \R^n$ be a domain. In \cite{DSpalphSphere} Francesca Da Lio and the author introduced $n$/$s$-harmonic maps as critical points of the energy
\begin{equation}\label{eq:FrancAJenergy}
 \int\limits_{\R^n} |\laps{s} u|^{p_s},\quad u: \Omega \to \S^{N-1},
\end{equation}
where $\S^{N-1} \subset \R^N$ is the unit sphere. In the critical case where $p_s = \frac{n}{s}$ we proved H\"older regularity; in the subcritical case where $p_s > \frac{n}{s}$ that kind of regularity follows from Sobolev embedding, in the supercritical case of $p_s < \frac{n}{s}$ one would not expect any kind of regularity for critical points without additional assumptions, see \cite{RiviereDisc}.  

To obtain regularity for critical points of \eqref{eq:FrancAJenergy} we extended the arguments used in the theory of fractional harmonic maps, critical points of
\begin{equation}\label{eq:Lapn4energy}
 \int\limits_{\R^n} |\lap^{\frac{n}{4}} u|^{2}, \quad u: \Omega \to \mathcal{N} \subset \R^N
\end{equation}
introduced in the pioneering work for $n=1$ by Da Lio and Rivi\`{e}re with the target $\mathcal{N}$ being a sphere \cite{DR1dSphere} and for general manifolds \cite{DR1dMan}, for extensions to higher order see \cite{SNHarmS10,DndMan,Sfracenergy}.

There are two drawbacks to considering \eqref{eq:FrancAJenergy}: On the one hand, although the energies look similar, they do not contain the classical case of $n$-harmonic maps, i.e. critical points to
\[
 \int\limits_{\R^n} |\nabla u|^{n} \quad u: \Omega \to \S^{N-1};
\]
and indeed the energy \eqref{eq:FrancAJenergy} can be treated in an easier way than the $n$-harmonic maps, since the term $|\laps{s} u|^{p_s-2}$ can simply be treated as a weight, and the arguments of \cite{DR1dSphere,SNHarmS10} otherwise go through without deeper changes. This was adressed in \cite{SLpGradient14}, where the author considered energies
\[
 \int\limits_{\R^n} |\Rz \laps{s} u|^{p_s}\quad u: \Omega \to \S^{N-1},
\]
where $\Rz = (\Rz_1,\ldots,\Rz_n)$ are the Riesz transform and this setup thus contains for $s=1$, $p_s = n$ the case of $n$-harmonic maps. 

Another drawback of \eqref{eq:FrancAJenergy} are applications to curvature energies: In \cite{BPSknot12} Blatt, Reiter and the author showed that one can extend the arguments of \cite{DR1dSphere,SNHarmS10} to the M\"obius energy \cite{OH91} and by this obtained regularity for critical points which before was known only for minimizers \cite{Freedman1994,He2000} by using the invariance under M\"obius transformations. One important ingredient to \cite{BPSknot12} is that the M\"obius energy, which has an integro-differential form, can be seen as an $L^2$-energy, \cite{Blatt2010a}, and this allowed us to use several arguments developed in \cite{DR1dSphere,SNHarmS10}. Looking at other critical curvature energies such as generalized versions of the tangent-point energy \cite{BlattReiterTangPoints}
and Menger curvature \cite{GonzalesMaddocksMenger}, one oberserves first that those are $L^{p}$-energies for possibly $p \neq 2$, however apart from scaling and differential order they seem to exhibit not too much similarities with \eqref{eq:FrancAJenergy}. In fact, they seem to be more related to the following energy, for $s \in (0,1)$ and $p_s = \frac{n}{s}$
\begin{equation}\label{eq:ourenergy}
 E_{s,p}(v) := \int \limits_{\Omega}\int \limits_{\Omega} \frac{\abs{v(x)-v(y)}^{p_s}}{\abs{x-y}^{n+s p}}\ dx\ dy, \quad \mbox{$v(x) \in \S^{N-1}$ a.e. in $\Omega$}.
\end{equation}
Observe that for $p_s = 2$, this virtually is the same as considering \eqref{eq:Lapn4energy}, but for $p_s \neq 2$ they are very different. Indeed, while \eqref{eq:Lapn4energy} considers the $L^2$-norm of $\lap^{\frac{n}{4}}$, \eqref{eq:ourenergy} can be interpreted as the Besov/Triebel-Lizorkin $\dot{B}^0_{p,p} = \dot{F}^0_{p,p}$-norm of $\laps{s} u$, which is much more diffult to handle. Up to now, there have been not enough techniques to treat critical geometric energies such as \eqref{eq:ourenergy}. Here, we obtain
\begin{theorem}\label{th:main:calcvar}
Assume that $u: \Omega \to \S^{N-1}$ is a critical point of \eqref{eq:ourenergy}, i.e. for any $\psi \in C_0^\infty(D,\R^N)$
\begin{equation}\label{eq:el}
 0 = \frac{d}{dt}\Big |_{t = 0} E_{s,p_s} \brac{\frac{u+t\psi}{|u+t\psi|}}.
\end{equation}
Then $u$ is H\"older continuous in $\Omega$.
\end{theorem}
Let us stress again, that the arguments needed for the proof of this theorem go beyond what was possible with the current techniques of fractional harmonic maps. One of the main problems is that $\laps{s} u$ may not be locally integrable, and thus it is difficult to obtain differential equations we can work with: The equations coming from the problem have the form of degenerate integro-differential equations. For example, the Euler-Lagrange equation becomes
\begin{proposition}[Euler-Lagrange Equations]
Any critical point as in Theorem~\ref{th:main} satisfies
\begin{equation}\label{eq:ourEL}
 \omega_{ij}\int\limits_{\Omega} \int\limits_{\Omega} \frac{|u(x)-u(y)|^{p_s-2} (u^i(x)-u^i(y)) (u^j(x) \varphi(x) - u^j(y) \varphi(y))}{|x-y|^{n+sp_s}}\ dx\ dy = 0 
\end{equation}
for any $\varphi \in C_0^\infty(\Omega)$ and any constant $\omega_{ij} = -\omega_{ji} \in \{-1,0,1\}$.
\end{proposition}
Theorem~\ref{th:main:calcvar} follows immediately from
\begin{theorem}\label{th:main}
Assume that $u: \Omega \to \R^N$, $|u| \equiv 1$ on $\Omega$ and $u$ is a solution to the integro-differential equation \eqref{eq:ourEL}. Then $u$ is H\"older continuous in $\Omega$.
\end{theorem}
One important tool in \cite{DR1dSphere,DR1dMan,SNHarmS10,DndMan,Sfracenergy} are estimates on three commutators
\begin{equation}\label{eq:def:Halpha}
 H_\alpha(a,b) := \laps{\alpha}(ab) - b\laps{\alpha}a -a\laps{\alpha}b,
\end{equation}
first introduced in \cite{DR1dSphere,DR1dMan}, see Theorem~\ref{th:bicomest}. In some sense $H_\alpha$ measures how far away the differential operator $\laps{\alpha}$ is from having a product rule. The intuition for $H_\alpha$ should come from classical operators $\alpha \in 2\N$, e.g.,
\[
 H_2(a,b) := 2 \nabla a\cdot \nabla b:
\]
The main ingredient to Theorem~\ref{th:bicomest} is that $H_\alpha$ behaves like a product of two differential operators of order less than $\alpha$ applied to $a$ and $b$, respectively.

This intuition, which leads to pointwise estimates for $H_\alpha$, \cite{Sfracenergy}, needs to be extended to our nonlinear $p/s$ situation, and they take the form
\begin{theorem}[Commutator Estimates]\label{th:threecomms}
Fix $s \in (0,1)$. For all $t < s$ large enough, let
\[
T_1(z) := \int \limits_{\R^n}\int \limits_{\R^n} \frac{|f(x)-f(y)|^{p_s-1}\ |\Gamma(x,y,z)|}{|x-y|^{n+sp_s}}\ dx\ dy,
\]
where
\[
 \Gamma(x,y,z) = |g(x) + g(y)-2g(z)|\ ||x-z|^{t-n} -|y-z|^{t-n}|.
\]
Then.
\[
\vrac{T_1}_{{\frac{n}{n-t}}} \aleq [f]_{\R^n,s,p_s}^{{p_s-1}}\ [g]_{\R^n,s,p_s} 
\]
Moreover let
\[
T_2 := \int \limits_{B_{\rho}}\int \limits_{B_{\rho}} \frac{|f(x)-f(y)|^{p_s-1}\ |\Theta(x,y)|}{|x-y|^{n+sp}}\ dx\ dy,
\]
where
\[
 \Theta(x,y) = \lapms{t} (g \laps{t} h)(x)-\lapms{t} (g \laps{t} h)(y) - \frac{1}{2} (h(x)-h(y)) (g(x)+g(y)).
\]
Then
\[ 
T_2 \aleq \vrac{\laps{t} g}_{\frac{n}{t}}\ [f]_{\R^n,s,p_s}^{p_s-1}\ [h]_{\R^n,s,p_s}.
\]
Here,
\[
 [f]_{\R^n,s,p} :=  \brac{\int\limits_{\R^n}\int\limits_{\R^n} \frac{|f(x)-f(y)|^{p}}{|x-y|^{n+sp}}}^{\frac{1}{p}}.
\]
\end{theorem}
We prove a localized version of Theorem~\ref{th:threecomms} in Lemma~\ref{la:thcommis:T1} and Lemma~\ref{la:thcommis:T2}.

As we learned from \cite{GiampieroPC}, the operator derived via variation from \eqref{eq:ourenergy} has recently also received attention \cite{GiampieroPaper} in the scalar setting with right-hand side zero, i.e. $u: \Omega \to \R$
\begin{equation}\label{eq:scalarcase}
 \frac{d}{dt}\Big |_{t = 0} E_{s,p} \brac{u+t\varphi} = 0 \quad \forall \varphi \in C_0^\infty(\Omega).
\end{equation}
To put our and their result into perspective, in the classical setting (i.e. what would be the $s = 1$ case) their result is related to regularity theory of
\[
 \operatorname{div}(|\nabla u|^{p-2} \nabla u) = 0,
\]
and they obtain H\"older continuity and forms of Harnack's inequality for general $p$ and $s$.
Our case, on the other hand, corresponds to the regularity theory of
\[
 \operatorname{div}(|\nabla u|^{p-2} \nabla u) = u|\nabla u|^p \in L^1(\R^n,\R^N).
\]
Our right-hand side is more complicated and our argument only treats the case $p = \frac{n}{s}$, but as mentioned above, in view of \cite{RiviereDisc} there is no hope of obtaining any sort of regularity for this kind of equation if $p < n$, and for $p> n$ H\"older continuity follows from Sobolev embedding. One might find that some of the arguments presented here can be used to obtain H\"older regularity for some relations of $s$ and $p$, just like $L^p$-theory for elliptic equations can be used to obtain such kind of regularity for equations $Lu = 0$.
We also like to mention a to a certain extend similar operator being introduced in the setting of fully nonlinear integro-differential equations in \cite{IN10,BCFNonlocalGrad}.

As explained above, we hope that in the spirit of \cite{BPSknot12} our arguments will be useful to obtain new regularity results for the geometric energies, such as the integral Menger curvature energies, see \cite{BlattReiterRegThMengerCurv,GonzalesMaddocksMenger,StrzSzvdMintegrMengerCurv,BlattNoteMenger}. 

Another extension of the arguments presented here might be used to treat the construction of minimizers in fixed homotopy classes via a flow argument, in the spirit of \cite[Chapter 6]{StruweVariationalMethods}. 

Finally let us mention that it seems not too difficult to extend our arguments to homogeneous spaces, but we leave this as an exercise, since our arguments here are already quite technical.

In the next section we give a sketch of the proof of Theorem~\ref{th:main}, which highlights the main arguments of the proof. The precise proof is given in Lemma~\ref{la:goal} after introducing some notation in Section~\ref{s:notation}. In Section~\ref{s:commis} we prove Theorem~\ref{th:threecomms}, in Lemma~\ref{la:thcommis:T1} and Lemma~\ref{la:thcommis:T2}. In Section~\ref{s:sobolev} we present a special version of Sobolev inequality which is crucial to our arguments. 

\section{Sketch of the proof}\label{s:sketch}
In this section we describe the main ideas of the proof. Some of the following arguments only make formal sense. Our goal is the following estimate: There is a $\tau \in (0,1)$ such that for any small Ball $B_\rho$,
\begin{equation}\label{eq:goal}
 \int\limits_{B_\rho}\int\limits_{B_\rho} \frac{|u(x)-u(y)|^{p_s}}{|x-y|^{n+sp_s}}\ dx\ dy\leq \tau \int\limits_{B_{5\rho}}\int\limits_{B_{5\rho}} \frac{|u(x)-u(y)|^{p_s}}{|x-y|^{n+sp_s}}\ dx\ dy + \mbox{good terms}. 
\end{equation}
The precise version of \eqref{eq:goal} is given in Lemma~\ref{la:goal}. Once \eqref{eq:goal} is obtained -- assuming the ``good terms''-part is behaving well -- the condition $\tau < 1$ allows us to iterate \eqref{eq:goal} for smaller and smaller balls, see the iteration arguments used in \cite{DR1dSphere,DR1dMan}, and also the presentation in \cite{BPSknot12}. 
\[
 \int\limits_{B_\rho}\int\limits_{B_\rho} \frac{|u(x)-u(y)|^{p_s}}{|x-y|^{n+sp_s}} \leq C_u\ \rho^\theta.
\]
where $\theta > 0$ depends heavily from $\tau < 1$. Then from Adams' \cite{Adams75} on Riesz potentials on Sobolev-Morrey spaces one obtains that $u$ belongs to a H\"older space. So indeed, \eqref{eq:goal} or more precisely Lemma~\ref{la:goal} imply Theorem~\ref{th:main}.

Now let us give an idea why \eqref{eq:goal} should be true. Given $B_\rho$, we set
\[
 [u]_{B_\rho} := \brac{\int\limits_{B_\rho}\int\limits_{B_\rho} \frac{|u(x)-u(y)|^{p_s}}{|x-y|^{n+sp_s}}dx dy}^{\frac{1}{p_s}}.
\]
Firstly, we can write for some smooth $\varphi$, compactly supported in, say, $B_{2\rho}$, $[\varphi] \leq 1$,
\[
 [u]_{B_\rho}^{p_s-1} \aleq \int\limits_{B_{3\rho}}\int\limits_{B_{3\rho}} \frac{|u(x)-u(y)|^{p_s-2}(u(x)-u(y))(\varphi(x)-\varphi(y))}{|x-y|^{n+sp_s}}dx dy + \ldots,
\]
where from now on we denote with ``$\ldots$'' good terms we don't want to focus on right now. 
We need to decompose the integro-differential term on the right-hand side. When $p_s = 2$ it corresponds essentially to
\[
 \int \laps{s} u\cdot \laps{s} \varphi,
\]
and we would still like to use this kind of representation, i.e. a weak PDE tested by a (pseudo-)derivative of $\varphi$. So using that, cf. \eqref{eq:rieszpotential},
\[
 \varphi (x) = c\int\limits_{\R^n} |x-z|^{s-n}\ \laps{s}\varphi(z)\ dz,
\]
we force the integro-differential equation to take a weak PDE-form:
\begin{equation}\label{eq:uvsTest}
 [u]_{B_\rho}^{p-1} \aleq \int \laps{s} \varphi(z)\ T_su(z)+ \ldots,
\end{equation}
where \[
 Tu^i(z) := 
\]
\[
\int \limits_{B_{2\rho}} \int \limits_{B_{2\rho}} \frac{|u(x)-u(y)|^{p_s-2}(u^i(x)-u^i(y))\ (|x-z|^{s-n} -|y-z|^{s-n})}{|x-y|^{n+sp}}\ dx\ dy.\nonumber
\]
Then, one would like to do the following estimate
\[
 [u]_{B_\rho}^{p-1} \aleq \vrac{\laps{s} \varphi}_{X^{p_s}(B_{3\rho})}\ \vrac{T_su}_{X^{p_s'}(B_{3\rho})}+ \ldots,
\]
where the space $X^p$ is chosen so that 
\begin{equation}\label{eq:estvarphileqTriebel1}
 \vrac{\laps{s} \varphi}_{X^{p_s}(B_{3\rho})} \aleq [\varphi]_{s,p_s,\R^n} \aleq 1.
\end{equation}
We will discuss the space $X^p$, and the problems that come with it, later, see Remark~\ref{rem:Xpsucks}. For simplicity , we pretend now that $X^p$ behaves similar to $L^p$. This is not true in general, and in Remark~\ref{rem:Xpsucks} we explain how we avoid this problem.

In this sense, we shall for now assume that
\[
 [u]_{B_\rho}^{p_s-1} \aleq \vrac{T_su}_{X^{p_s'}(B_{3\rho})}+ \ldots.
\]

We use the following decomposition which is true since $|u| = 1$.
\begin{equation}\label{eq:LagrangeEst}
 |\overrightarrow{v}|_{\R^N} \aleq |u^i(x)\ \overrightarrow{v}^i| + \max_{\omega_{ij}} |u^j(x) \omega_{ij} \overrightarrow{v}^i|\quad \mbox{for any $\overrightarrow{v} \in \R^N$}.
\end{equation}
The maximum is taken over all matrices $\omega \in \{-1,0,1\}^{N \times N}$ with $\omega_{ij} = - \omega_{ji}$. Note that these are finitely many $\omega_{ij}$. Moreover, we recall that we use Einstein's summation convention. \eqref{eq:LagrangeEst} can be seen as a consequence of what sometimes is called the Lagrange-identity. A direct proof for \eqref{eq:LagrangeEst} can be found in, e.g., \cite{DSpalphSphere}. This kind of decomposition has been used in this form in \cite{SNHarmS10}, motivated from a very similar decomposition in \cite{DR1dSphere}.

So we have to estimate
\begin{equation}\label{eq:sketch:thedecompdest}
 [u]_{B_\rho}^{p_s-1} \aleq \vrac{u^i T_su^i}_{X^{p_s'}(B_{3\rho})}+ \vrac{u^j\omega_{ij} T_su^i}_{X^{p_s'}(B_{3\rho})} + \ldots.
\end{equation}
For the first term (which measures the part of $T_su$ which is orthogonal to the sphere) we use that $(u(x)-u(y)) \cdot (u(x)+u(y)) = |u|^2(x)-|u|^2(y) = 0$ for $x,y \in B_{2\rho}$ and have
\[u^i(z) T_su^i(z) \equiv u(z) \cdot T_su(z) =\]
 \[
 \int \limits_{B_{2\rho}} \int \limits_{B_{2\rho}} \frac{|u(x)-u(y)|^{p_s-2} (u(x)-u(y))\cdot \Gamma(x,y,z) (|x-z|^{s-n} -|y-z|^{s-n})}{|x-y|^{n+sp}}\ dx\ dy,
 \]
where
\[
 \Gamma(x,y,z) := -\frac{1}{2} (u(x)+u(y)-2u(z)).
\]
This means that in some sense $u(z) \cdot T_su(z)$ can be interpreted as a product of lower-oder operators, in view of Theorem~\ref{th:threecomms}. The precise, localized estimates are given in Lemma~\ref{la:thcommis:T1}. To give the reader an intuition why this should be true, let us motivate this effect by the following completely unprecise argument: If we interpret $[u]_{s,p_s,\Omega} < \infty$ as 
\[
 \frac{u(x)-u(y)}{|x-y|^{s}} \quad \mbox{is ``well integrable''},
\]
then the term $\Gamma(x,y,z)$ (in a very unprecise sense) also
\[
 \frac{\Gamma(x,y,z)}{|x-y|^{s}} \quad \mbox{is ``well integrable''}.
\]
This means that 
\[
 \frac{|u(x)-u(y)|^{p_s-2} (u(x)-u(y))\cdot \Gamma(x,y,z)}{|x-y|^{sp}} \quad \mbox{is ``well integrable''},
\]
and since then $|x-z|^{s-n}$ (the kernel of the operator $\lapms{s}$) is left over, in that very unprecise sense,
\[
 u \cdot T_s u \mbox{ ``$\approx$'' } \lapms{s} |\laps{s}u|^p,
\]
so in an even more unprecise sense by a formal Sobolev-inequality argument
\[
 \vrac{u \cdot T_s u}_{X^{p'}} \aleq \vrac{\laps{s}u}^p_{X^p} \aleq [u]_{B_{5\rho}}^{p_s} + \ldots.
\]
Now using that by absolute continuity of integrals $[u]_{B_{5\rho}} < \delta$ on all small balls, this becomes
\begin{equation}\label{eq:ucdotTsuest}
 \vrac{u \cdot T_s u}_{X^{p'}(B_{3\rho})} \aleq \delta [u]_{B_{5\rho}}^{p_s-1} + \ldots.
\end{equation}
The precise argument leading to \eqref{eq:ucdotTsuest} is given in Lemma~\ref{la:orthogonal}.

It remains to estimate
\begin{equation}\label{eq:uomegaTsuest}
 \max_{\omega} \vrac{\omega_{ij}u^j T_s u^i}_{X^{p'}(B_{3\rho})} \aleq \delta [u]_{B_{5\rho}}^{p_s-1} + \ldots.
\end{equation}
In the precise form this is done in Lemma~\ref{la:tangential}. The formal idea is as follows: Fix $\omega$. Firstly, we estimate this by a PDE, for some smooth $\varphi$, compactly supported in $B_{5\rho}$ and $[\varphi]\leq1$,
\[
 \vrac{\omega_{ij}u^j T_s u^i}_{X^{p'}(B_{3\rho}} \aleq \int (\laps{s} \varphi) \omega_{ij}u^j T_s u^i + \ldots=\\
 \]
\[ -\int \laps{s} (\varphi \omega_{ij}u^j) T_s u^i  -\int \varphi \omega_{ij}(\laps{s} u^j) T_s u^i  -\omega_{ij}\int H_s(\varphi,u^j) T_s u^i + \ldots\]
where $H_s$ is from \eqref{eq:def:Halpha}. The estimates on $H_s$ have been already used in the fractional harmonic map case (i.e., the $L^2$-case), and also here it can be dealt with in a more subtle yet similar fashion, using Theorem~\ref{th:bicomest}. For the remaining parts we firstly use again that $\lapms{s} \laps{s} f = f$, basically inverting the argument in \eqref{eq:uvsTest},
\[
 \omega_{ij}\int \laps{s} (\varphi u^j) T_s u^j =\] 
 \[\omega_{ij}\int\limits_{\Omega} \int\limits_{\Omega} \frac{|u(x)-u(y)|^{p_s}(u(x)-u(y))(\varphi u^j(x)-\varphi u^j(y))}{|x-y|^{n+sp_s}} dx dy + \ldots
\]
This is where the Euler-Lagrange equation \eqref{eq:ourEL} comes into effect and sets the right-hand side zero for arbitrary constant $\omega \in \{-1,0,1\}^{N \times N}$, if it is only antisymmetric.

Lastly, we need to treat
\[
 \int \varphi\ \omega_{ij}\ (\laps{s} u^j)\ T_s u^i.
\]
In the classical or even ($n$/$s$-)fractional harmonic map setting this term is zero since $\omega$ is antisymmetric and essentially $T_s u^i = \laps{s} u^i$. This is not true anymore in the integro-differential case, and we write this term as
\[
  \int \limits_{B_{2\rho}} \int \limits_{B_{2\rho}} \frac{|u(x)-u(y)|^{p_s-2} (u^i(x)-u^i(y))\Theta^i(x,y)}{|x-y|^{n+sp}}\ dx\ dy,
\]
with
\[
\Theta^i(x,y) = \lapms{s}(\varphi \omega_{ij}(\laps{s} u^j))(x)-\lapms{s}(\varphi \omega_{ij}(\laps{s} u^j))(y).
\]
This time we use that by the antisymmetry of $\omega$.
\[
 (u^i(x)-u^i(y)) \omega_{ij} (u^j(x)-u^j(y))(\varphi(x)+\varphi(y)) = 0.
\]
Then we can replace $\Theta^i(x,y)$ by
\[
 \omega_{ij} (\lapms{s}(\varphi (\laps{s} u^j))(x)-\lapms{s}(\varphi (\laps{s} u^j))(y) - \frac{1}{2} (u^j(x)-u^j(y))(\varphi(x)+\varphi(y)).
\]
This term again falls under the ``commutator estimate'', Theorem~\ref{th:threecomms}. The main observation is then that
\[
 \Theta^i(x,y) = -\frac{1}{2} \int \limits_{\R^n} (|x-z|^{t-n} -|y-z|^{t-n})\ \laps{s}u(z)\ (\varphi(x)+\varphi(y)-2\varphi(z))\ dz,
\]
and again the intuition should be that this term can ``integrate well'' against $|x-y|^{-2s}$. For the precise statement and proof of the above representation see, Lemma~\ref{la:thcommis:T2}. Then in the same formal way as above for \eqref{eq:ucdotTsuest}, more precisely see Lemma~\ref{la:tangential}, we obtain \eqref{eq:uomegaTsuest}.

The estimates \eqref{eq:ucdotTsuest} and \eqref{eq:uomegaTsuest} plugged into \eqref{eq:sketch:thedecompdest} then imply \eqref{eq:goal}.
\begin{remark}\label{rem:Xpsucks}
It turns out that the right space $X^{p_s}$ for estimate \eqref{eq:estvarphileqTriebel1} is the homogeneous zero-order Triebel-Lizorkin space $\dot{F}^0_{p_s,p_s}$, see e.g. \cite{Triebel1983,GrafakosMF}. But this space leads to problems if $p_s > 2$: to the best of our knowledge, it is not necessarily true $X^p \subset L^1_{loc}$ or $L^\infty \cdot X^p \subset X^p$. Moreover $f \leq g$ does not necessarily imply that $\vrac{f}_{X^p} \leq \vrac{g}_{X^p}$. This makes $X^p$ unsuitable for our estimates, for example for using the pointwise estimate \eqref{eq:LagrangeEst} in order to obtain \eqref{eq:sketch:thedecompdest}.

Our solution to this problem is to estimate expressions \emph{below the natural differentiation order}, i.e. we consider for $t < s$
\[
 [u]_{B_\rho}^{p-1} \aleq \int \laps{t} \varphi\ T_t u+ \ldots,
\]
which can be shown to be true for $t$ less than but sufficiently close to $s$. The gain is that although 
\[
 \vrac{\laps{s} \varphi}_{L^{\frac{n}{s}}} \not \aleq [\varphi]_{s,p_s,\R^n},
\]
we still have a version of Sobolev embedding, stated in Theorem~\ref{th:sobolev:new}, such that for $t < s$,
\[
 \vrac{\laps{t} \varphi}_{L^{\frac{n}{t}}} \aleq [\varphi]_{s,p_s,\R^n}.
\]f
In other words, $\laps{t} u \in L^1_{loc}$ for $t < s$, a fact which fails for $\laps{s} u$.
On the other hand, this creates a whole new set of technical difficulties, so we skipped this problem and for this sketch take $t=s$, and pretended as if all the needed $L^p$-properties for $X^p$ were true. 
\end{remark}
\section{Preliminaries and Notation}\label{s:notation}
The fractional Laplacian $\laps{t} f$ for $t \in (0,1)$ and $f$ in the Schwartz class $\Sw(\R^n)$ is given by
\begin{equation}\label{eq:lapsalpha}
 \laps{t} f(x) = c_{t} \int \limits_{\R^n} \frac{f(y)-f(x)}{|x-y|^{n+t}}\ dy.
\end{equation}
The inverse of $\laps{t}$, the Riesz potential, is denoted by $\lapms{t}$, and is given by
\begin{equation}\label{eq:rieszpotential}
 \lapms{t} F (x) = \tilde{c}_t\ \int \abs{y-x}^{t-n}\ F(y)\ dy.
\end{equation}
For more details and arguments on these operators and related norms, we refer to, e.g., \cite{SKM93,Landkof72,Tartar07,TaylorIII,Hitchhiker}.
We recall, that we have the exponents
\[
 p_t := \frac{n}{t}, \quad p_s := \frac{n}{s}.
\]
With $\aleq$, $\aeq$ $\ageq$ we mean $\leq$, $=$, $\geq$ up to a multiplicative constant, always possibly depending on $s$, $t$, $n$, $N$. With $C$ we denote a generic constant, which may change from line to line.

For a ball $B$ and $t > 0$, we denote
\begin{equation}\label{eq:defTBt}
 T_{B,t} u^i(z) := 
\end{equation}
\[
\int \limits_{B} \int \limits_{B} \frac{|u(x)-u(y)|^{p_s-2}(u^i(x)-u^i(y))\ (|x-z|^{t-n} -|y-z|^{t-n})}{|x-y|^{n+sp}}\ dx\ dy.\nonumber
\]
This quantity does not make much sense if $t$ is too small. But if $t$ is almost or larger than $s$, more precisely $t > 1-(1-s)p_s$ it is well-defined at least for $u$ belonging to the Schwartz class $\mathcal{S}(\R^n)$. It is our standing assumption throughout the paper that whenever we work with $T_{B,t} u$ we restrict our attention to such a $t$. It is crucial to check that there exists an admissible $t < s$. Its interest to us stems from the following identity which holds for any $\varphi \in C_0^\infty(\R^n)$
\begin{equation}\label{eq:dualargTBt}
 \int \limits_{\R^n} \laps{t} \varphi\ T_{B,t} u^i = c\ \int \limits_{B} \int \limits_{B} \frac{|u(x)-u(y)|^{p_s-2}(u^i(x)-u^i(y))\ (\varphi(x)-\varphi(y))}{|x-y|^{n+sp}}\ dx\ dy.
\end{equation}
Indeed, this follows from the definition of $\lapms{t}$ and the fact that $\lapms{t} \laps{t}$ is the identity. With this weak definition of $T_{B,t}$ at hand, one can also check
\begin{equation}\label{eq:lapmsTbt}
 \lapms{\tilde{t}} T_{B,t} u^i(z)  := T_{B,t+\tilde{t}} u^i(z).
\end{equation}
The semi-norm we are going to estimate for $t \in (0,1)$ is the following for a ball $B \subset \R^n$
\[
 [f]_{t,p,B} := \int \limits_{B} \int \limits_{B} \frac{|f(x)-f(y)|^p}{|x-y|^{n+pt}}\ dx\ dy.
\]
Obviously,
\[
 [f]_{t,p,B} \leq [f]_{t,p,\tilde{B}} \quad \mbox{if $B \subset \tilde{B}$}.
\]
We need to work with several cutoff functions. We have the mollified ones, denoted by $\eta$, and the index-cutoff, denoted by $\chi$. If the characteristic-function is cutting off a ball, we will denote it by $\scut$, $\cut$, and if it cuts of an annulus, we write $\scutA$, $\cutA$. We fix a ball $B_R(x_0)$, and index the cutoff-functions according to their relation with $B_R(x_0)$. More precisely, we use the following
\begin{definition}[Cutoff functions]\label{def:cutoffs}
For a fixed ball $B_R(x_0)$, we define the following:

If $\chi_A$ is the characteristic function on $A$, we denote for $l \in \Z$,
\[
 \cut_{l} := \chi_{B_{2^lR}(x_0)}, \mbox{ and } \cutA_l := \cut_l - \cut_{l-1}.
\]
The mollified version of these cutoffs are denoted by $\scut_l$ so that
\[
 \scut_l \in C_0^\infty(B_{2^{l+1}R}(x_0)), \quad \scut_l \equiv 1 \mbox{ on $B_{2^lR}(x_0)$} \quad |\nabla^i \scut_l| \aleq (2^l R)^{-i},
\]
and
\[
 \scutA_l := \scut_l - \scut_{l-1}.
\]
If the scale ball $B_R(x_0)$ is clear, we will often write
\[
 B_l := B_{2^l R}(x_0).
\]
We will denote the mean value
\[
 (f)_{l} := |B_{l}|^{-1}\ \int\limits_{B_l} f,
\]
and
\[
 [f]_l := [f]_{s,p_s,B_{l}}.
\]
Also, we will denote by
\[
 [f]_\infty = [f]_{s,p_s,\R^n}.
\]
\end{definition}

\section{The Main Lemma}
Theorem~\ref{th:main} follows from the following Lemma
\begin{lemma}\label{la:goal}
There exists a $\tau \in (0,1)$, $\sigma > 0$, $L_0 \in \N$, $\rho_0 > 0$, such that for any $B_\rho(x_0) \subset \Omega$, $\rho < \rho_0$ and any $L \geq L_0$ such that $B_{2^L \rho}(x_0) \subset \Omega$, we have for $p_s = \frac{n}{s}$
\[
[u]_{s,p_s,B_\rho(x_0)}^{p_s} \leq \tau [u]_{s,p_s,B_{2^L\rho}(x_0)}^{p_s} + \sum_{l=1}^\infty 2^{-\sigma(L+l)}\ [u]_{s,p_s,B_{2^{L+l}\rho}(x_0)}^{p_s}.
\]
\end{lemma}
From Lemma~\ref{la:goal} we obtain the proof of Theorem~\ref{th:main} as described in Section~\ref{s:sketch}.
\begin{proof}[Proof of Lemma~\ref{la:goal}]
By an extension argument we may assume that $u$ is defined everywhere on $\R^n$, $u \in L^p \cap L^\infty(\R^n)$, and that
\[
 [u]_{s,p_s,\R^n} < \infty.
\]
For some $\delta > 0$ to be determined later, let $\rho_0 > 0$ be so that
\begin{equation}\label{eq:lamain:usmalsmall}
 \sup_{\rho < \rho_0, x_0 \in \R^n} [u]_{s,p_s, B_\rho(x_0)} < \delta.
\end{equation}
Such a $\rho_0$ exists by absolute continuity of the integrals.

For simplicity of presentation, we assume $\rho_0 = 1$ and show then only the claim for $B_1(0)$: fix the basic scale ball $B_R(x_0)$ from Definition~\ref{def:cutoffs} as $B_1(0)$, and assume that $B_{2^{L_0}}(0) \subset \Omega$ for a huge $L_0 \in \N$, where $L_0$ is determined from the applications of the follwing Lemmas. 
We define
\[
 \operatorname{Tail}(\sigma,L,C) := C\ \sum_{l=1}^\infty 2^{-\sigma(L+l)}\ [u]_{s,p_s,L+l}^{p_s}.
\]
The claim of Lemma~\ref{la:goal} takes the form
\begin{equation}\label{eq:goaleq}
 [u]_{0}^{p_s} \leq \tau [u]_{L}^{p_s} + \operatorname{Tail}(\sigma,L,C).
\end{equation}
Note that for any $\eps > 0$, if $L$ is large enough (depending on $\sigma$ and $C$), we have
\begin{align*}
 \operatorname{Tail}(\sigma,L,C) &\leq \eps [u]_{\tilde{L}}^{p_s} +  C\ \sum_{l=1}^\infty 2^{-\sigma(\tilde{L}+l)}\ [u]_{s,p_s,\tilde{L} +l}^{p_s}\\
 &\leq \eps [u]_{\tilde{L}}^{p_s} +  \operatorname{Tail}(\sigma,\tilde{L},C).
 \end{align*}
This means that the tail can be shifted from $L$ to $\tilde{L} > L$ without doing much harm in terms of obtaining \eqref{eq:goaleq}. In the following we thus consider $\sigma$, $C$, and even $L$ a constant that can increase (in the case of $C$, $L$) or decrease (in the case of $\sigma$) as the proof progresses.

The first step for \eqref{eq:goaleq} is Lemma~\ref{la:lhsest}. Let $K > 0$ and $L = 10K$,
\[
[u]_{0}^{p_s}
\leq (\eps + C\ 2^{-K\sigma})\ [u]_{L}^{p_s}  + C_\eps \brac{[u]_{L}^{p_s} - [u]_{0}^{p_s}}+C\ [u]_1\ \vrac{\cut_K(z) T_{B_L,t} u}_{p_t'}.\\
\]
We now follow the so-called Widman-holefilling trick: Add $C_\eps [u]_{B_{0}}^{p_s}$ to both sides and divide by $C_\eps + 1$. Then
\[
 [u]_{0}^{p_s} \leq \frac{\eps + C\ 2^{-K\sigma}+C_\eps}{C_\eps + 1} [u]_{L}^{p_s} +\frac{C}{C_\eps + 1}\ [u]_1\ \vrac{\cut_K(z) T_{B_L,t} u}_{p_t'}.\\
\]
Taking $K$ large enough, and $\eps$ small enough, so that $\eps + C\ 2^{-K\sigma} < 1$. Then,
\[
 \tau := \frac{\eps + C\ 2^{-K\sigma}+C_\eps}{C_\eps + 1} < 1,
\]
and we have
\begin{equation}\label{eq:lamain:firstest}
 [u]_{0}^{p_s} \leq \tau [u]_{L}^{p_s} +C\ [u]_1\ \vrac{\cut_K(z) T_{B_L,t} u}_{p_t'}.\\
\end{equation}
We know that $\cut_K |u| = \cut_K$, because we assume that $B_L \subset \Omega$ and $u(\Omega) \subset \S^{N-1}$. Thus \eqref{eq:LagrangeEst} is applicable.

We obtain by Lemma~\ref{la:orthogonal},
\[
 [u]_1\ \vrac{\cut_K(z) u^i T_{B_L,t} u^i}_{p_t'} \aleq [u]_1\ [u]_{2L}^{p_s} + [u]_1\ \operatorname{Tail}(\sigma,L,C).
\]
and by Lemma~\ref{la:tangential}
\begin{align*}
 [u]_1\ \vrac{\cut_K(z) u^j \omega_{ij} T_{B_L,t} u^i}_{p_t'} &\aleq [u]_1[u]_{L}^p + [u]_1 2^{-\sigma L}   [u]_{L}^{p-1} + [u]_1 [u]_{\infty}\ \sum_{k=1}^\infty 2^{-\sigma(L+k)} [u]_{L+k}^{p-1}\\
 &\aleq [u]_1\ [u]_{L}^p + 2^{-\sigma K}   [u]_{L}^{p_s} + [u]_{\infty}\ \sum_{k=1}^\infty 2^{-\sigma(L+k)} [u]_{L+k}^{p_s}.
\end{align*}
In view of this, \eqref{eq:lamain:firstest} becomes
\[
 [u]_{0}^{p_s} \leq \tau [u]_{L}^{p_s} +C ([u]_1+2^{-\sigma K})[u]_{2L}^{p_s}+ \operatorname{Tail}(\sigma,L,C+[u]_\infty+[u]_1).
\]
Taking $K$ large enough, and $\delta > 0$ from \eqref{eq:lamain:usmalsmall} small enough, there is $\tilde{\tau} \in (\tau,1)$, so that
\[
 [u]_{0}^{p_s} \leq \tilde{\tau} [u]_{2L}^{p_s} + \operatorname{Tail}(\sigma,L,\tilde{C}).
\]
This proves Lemma~\ref{la:goal}.
\end{proof}

\section{On the relation of the semin-norm \texorpdfstring{$[\cdot]$}{} to \texorpdfstring{$T_{B,t}$}{T}}
Recall our conventions from Definition~\ref{def:cutoffs} applied to $B_1(0)$, and the Definition of $T_{B,t}$ \eqref{eq:defTBt}.
\begin{lemma}\label{la:lhsest}
For any $\eps > 0$, $K,L \in \N$, $0 < t < s$, and $p_{t} = \frac{n}{t}$,
\begin{align*}
[u]_{0}^{p_s}
&\leq (\eps + C\ 2^{-K\sigma})\ [u]_{L}^{p_s}\\
& + C_\eps \brac{[u]_{L}^{p_s} - [u]_{0}^{p_s}}\\
&+C\ [u]_1\ \vrac{\cut_K(z) T_{B_L,t} u}_{{p_t}'}.\\
\end{align*}
\end{lemma}
\begin{proof}
Recall that from Definition~\ref{def:cutoffs}, $\scut_0 \equiv 1$ in $B_0$. Denoting
\[
 \psi(x) := \scut_0(x) (u(x)-(u)_{0}),
\]
we have
\[
 [u]_{0}^{p_s} \leq \int \limits_{B_{L}}\int \limits_{B_{L}} \frac{|u(x)-u(y)|^{p_s-2} (\psi(x)-\psi(y))\ (\psi(x)-\psi(y))}{|x-y|^{n+sp_s}} dx\ dy
\]
Now we write
\begin{align*}
 \psi(x) - \psi(y) =&(u(x)-u(y)) - (1-\scut_0(x)) (u(x)-u(y))\\
 & + (\scut_0(x)-\scut_0(y)) (u(y)-(u)_{0}).
\end{align*}
so using that $\scut_0 \equiv 1$ on $B_0$,
\[
 [u]_0^{p_s} \aleq I + II + III,
\]
where
\[
 I :=\int \limits_{B_{L}}\int \limits_{B_{L}} \frac{|u(x)-u(y)|^{p_s-2} (u(x)-u(y))(\psi(x)-\psi(y))}{|x-y|^{n+sp_s}} dx\ dy,
 \]
 \[
 II := \int \limits_{B_{L}}\int \limits_{B_{L} \backslash B_{0}} \frac{|u(x)-u(y)|^{p_s-1} |\psi(x)-\psi(y)|}{|x-y|^{n+sp_s}} dx\ dy,
 \]
 and using that $\scut_0(x) - \scut_0(y) = 0$ if both both $x,y \in B_0$,
\begin{align*}
 III \aleq  &\int \limits_{B_{L}\backslash B_0}\int \limits_{B_{L}} \frac{|u(x)-u(y)|^{p_s-2} |\scut_0(x)-\scut_0(y)| |u(y)-(u)_{0}||\psi(x)-\psi(y)|}{|x-y|^{n+sp_s}} dx\ dy\\
 &+\int \limits_{B_{L}}\int \limits_{B_{L}\backslash B_0} \frac{|u(x)-u(y)|^{p_s-2} |\scut_0(x)-\scut_0(y)| |u(y)-(u)_{0}||\psi(x)-\psi(y)|}{|x-y|^{n+sp_s}} dx\ dy.
\end{align*}
Since
\[
 |\psi(x)-\psi(y)| \leq |\scut_0(x)-\scut_0(y)|\ |u(y)-(u)_0| + |u(x)-u(y)|.
\]
we have for $X = (B_L \backslash B_0 \times B_L) \cup (B_L \times B_L \backslash B_0)$
\begin{align*}
 II + III \aleq &\int\int\limits_{X} \frac{|u(x)-u(y)|^{p_s-2} |\scut_0(x)-\scut_0(y)|^2 |u(y)-(u)_{0}|^2}{|x-y|^{n+sp_s}} dx\ dy\\
 &+\int\int\limits_{X} \frac{|u(x)-u(y)|^{p_s-1} |\scut_0(x)-\scut_0(y)| |u(y)-(u)_{0}|}{|x-y|^{n+sp_s}} dx\ dy\\
 &+\int\int\limits_{X} \frac{|u(x)-u(y)|^{p_s}}{|x-y|^{n+sp_s}} dx\ dy.
\end{align*}
Using H\"older's inequality, Proposition~\ref{pr:widmanguyest}, and Proposition~\ref{pr:etavsummvest}, and then Young's inequality for any $\eps > 0$,
\[
 II + III \aleq C_\eps ([u]_{L}^{p_s} - [u]_{0}^{p_s}) +  \eps [u]_{L}^{p_s}.
\]
It remains to treat $I$, where by Proposition~\ref{pr:psiest}\footnote{using also the density of smooth functions in the space with bounded $[u]_{1}$, which follows from related results in Triebel spaces}
\[
I \aleq [u]_{1} \sup_{\varphi \in C_0^\infty(B_{1}), [\varphi]_{\infty} \leq 1} \int \limits_{B_{L}}\int \limits_{B_{L}} \frac{|u(x)-u(y)|^{p_s-2} (u(x)-u(y))(\varphi(x)-\varphi(y))}{|x-y|^{n+sp_s}} dx\ dy.
\]
We conclude by the following Lemma~\ref{la:lhsestremaining}.
\end{proof}

\begin{lemma}\label{la:lhsestremaining}
Fix $0 < t < s$ close enough to $s$, and ${p_t} = \frac{n}{t}$. Then for any $L, K \in \N$,
\[
\sup_{\varphi \in C_0^\infty(B_{1}), [\varphi]_\infty\leq1} \int \limits_{B_L} \int \limits_{B_L} \frac{|u(x)-u(y)|^{p-2}(u(x)-u(y)) \cdot (\varphi(x) -\varphi(y))}{|x-y|^{n+sp}}\ dx\ dy
\]
\[\aleq \vrac{\cut_{K+1} T_{B_L,t} u}_{{p_t}'} + 2^{-K}[u]_{L}^{p_s-1}. \]
\end{lemma}
\begin{proof}
We use \eqref{eq:dualargTBt}, and need to estimate
\[
I:=   \int \limits_{\R^n} \scut_K(z) \laps{t} \varphi(z)\ T_{B_L,t} u(z)\ dz\\
\]
and
\[
II :=  \sum_{k=1}^\infty \int \limits_{\R^n} \scutA_{K+k}(z) \laps{t} \varphi(z)\ T_{B_L,t} u.
\]
As for $I$,
\[
|I| \aleq \vrac{\laps{t} \varphi}_{{p_t}}\ \vrac{\scut_K T_{B_L,t} }_{({p_t})'}
\]
and by Theorem~\ref{th:sobolev:new},
\[
 \vrac{\laps{t} \varphi}_{{p_t}} \aleq [\varphi] \aleq 1.
\]
The remaining term $II$ is treated as follows
\begin{align*}
II=&\sum_{k =1}^\infty \int \limits_{\R^n} \laps{2s-t}(\scutA_{K+k}(z) \laps{t} \varphi)(z)\ \lapms{2(s-t)}T_{B_L,t}u(z) dz\\
\overset{\eqref{eq:lapmsTbt}}{=}&\sum_{k =1}^\infty \int \limits_{\R^n} \laps{2(t-s)}(\scutA_{K+k}(z) \laps{t} \varphi)(z)\ T_{B_L,2s-t}u(z) dz\\
\aleq& \sum_{k =1}^\infty \vrac{\laps{2(t-s)}(\scutA_{K+k} \laps{t} \varphi)}_{\frac{n}{2s-t}}\ \vrac{T_{B_L,2s-t}u}_{\frac{n}{n-2s+t}}\\
\end{align*}
Proposition~\ref{pr:lapmsTest}, for $\delta = s-t> 0$ small enough,
\[
 \vrac{T_{B_L,2s-t}u}_{\frac{n}{n-2s+t}} \aleq [u]_{L}^{p_s-1}.
\]
Proposition~\ref{pr:localizationT1} implies that
\[
 \vrac{\laps{2(t-s)}(\scutA_{K+k} \laps{t} \varphi)}_{\frac{n}{2s-t}} \aleq 2^{-\sigma(K+k)}\ \vrac{\laps{t} \varphi}_{p_t} \aleq 2^{-\sigma(K+k)}.
\]
\end{proof}

\section{Estimates on \texorpdfstring{$T_{B,s}u$}{Tu}}
In this section we show in Lemma~\ref{la:orthogonal} and Lemma~\ref{la:tangential} how the $Tu$ decomposed in $u^i Tu^i$ and $u^i \omega_{ij} Tu^j$ can be estimated.
\subsection{Orthogonal Part}
\begin{lemma}\label{la:orthogonal}
Assume that $\cut_{L}|u| = \cut_{L}$, $L \in \Z$, then for some $\sigma > 0$
\[
 \vrac{u^iT_{B_L,t} u^i}_{p_t'} \aleq [u]_{2L}^{p_s} + \sum_{l=1}^\infty 2^{-\sigma(L+l)} [u]_{2L+l}^{p_s}.
\]
\end{lemma}
\begin{proof}
For any $x,y \in B_{L}$ we have
\[
 (u^i(x) - u^i(y))\ (u^i(x) + u^i(y)) = |u|^2(x) - |u|^2(y) = 1-1 = 0.
\]
Thus,
\begin{align*}
 &|u^i(z)\int \limits_{B_{L}} \int \limits_{B_{L}} \frac{|u(x)-u(y)|^{p-2}(u^i(x)-u^i(y)) (|x-z|^{t-n} -|y-z|^{t-n})}{|x-y|^{n+sp}}\ dx\ dy |\\
 \aleq& \int \limits_{B_{L}} \int \limits_{B_{L}} \frac{|u(x)-u(y)|^{p-1}\ |u(x) + u(y)-2u(z)|\ ||x-z|^{t-n} -|y-z|^{t-n}|}{|x-y|^{n+sp}}\ dx\ dy\\
\end{align*}
Now the claim follows by Lemma~\ref{la:thcommis:T1}.
\end{proof}

\subsection{Tangential part: Application of the Euler-Lagrange equation}
\begin{lemma}\label{la:tangential}
For any $K \in \Z$, if $B_{30K} \subset \Omega$ and $u$ satisfies \eqref{eq:ourEL}. If $t < s$ is close enough to $s$, then for some $\sigma > 0$,
\begin{align*}
&\vrac{\cut_K  \omega_{ij} u^j T_{t,B_{10K}} }_{p_t'} \aleq &[u]_{20K}^{p_s} + 2^{-\sigma K}   [u]_{20K}^{p_s-1} + [u]_{\infty}\ \sum_{k=1}^\infty 2^{-\sigma(K+k)} [u]_{20K+k}^{p_s-1}
\end{align*}
\end{lemma}
\begin{proof}
Let $L = 10K$.
We have for some $g \in L^{p_t}$
 \[\vrac{\cut_K  \omega_{ij} u^j   T_{B_L,t} u^i}_{p_t'}
 \aleq \int (\cut_K g)\ \omega_{ij} u^j\  T_{B_L,t} u^i = I + II,
 \]
where, using again $f = \laps{t}\lapms{t} f$,
\begin{align*}
 I:=&\int \laps{t} (\scut_{2K} \lapms{t} (\cut_K g))\ \omega_{ij} u^j\  T_{B_L,t} u^i,\\
 II:=&\sum_{k=1}^\infty \int \laps{t} (\scutA_{2K+k} \lapms{t} (\cut_K g))\ \omega_{ij} u^j\  T_{B_L,t} u^i.
\end{align*}
As for $II$, (we make sure that $s < 2s-t < 1$)
\begin{align*}
 &\int \laps{2s-t}( (\laps{t} (\scutA_{2K+k} \lapms{t} (\cut_K g))\ \omega_{ij} u^j))\  \lapms{2(s-t)}T_{B_L,t} u^i\\
 \overset{\eqref{eq:lapmsTbt}}{=}&\int \laps{2(s-t)}( (\laps{t} (\scutA_{2K+k} \lapms{t} (\cut_K g))\ \omega_{ij} u^j))\  T_{B_L,2s-t} u^i\\
 \aleq& \vrac{\laps{2(s-t)}( (\laps{t} (\scutA_{2K+k} \lapms{t} (\cut_K g))\ \omega_{ij} u^j))}_{\frac{n}{2s-t}}\  \vrac{T_{B_L,2s-t} u^i}_{\frac{n}{n-2s+t}}\\
 \aleq& \vrac{\laps{2(s-t)}( (\laps{t} (\scutA_{2K+k} \lapms{t} (\cut_K g))\ \omega_{ij} u^j))}_{\frac{n}{2s-t}}\  [u]_{B_L}^{p-1}\\
\end{align*}
In the last step we used Proposition~\ref{pr:lapmsTest}.

It remains to estimate 
\[
 \vrac{\laps{2(s-t)}( (\laps{t} (\scutA_{2K+k} \lapms{t} (\cut_K g))\ \omega_{ij} u^j))}_{\frac{n}{2s-t}}.
\]
By Proposition~\ref{pr:localizationT1},
\[
 \vrac{\laps{t+\delta}(\scutA_{K+L} \lapms{t} (\cut_K g)}_{\frac{n}{t+\delta }} \aleq 2^{-(K+k)\frac{n}{p_t}} \vrac{g}_{p_t}.
\]
Moreover, we assumed w.l.o.g $\vrac{u}_\infty \leq 1$, so
\begin{align*}
& \vrac{\laps{2(s-t)}( (\laps{t} (\scutA_{2K+k} \lapms{t} (\cut_K g))\ \omega_{ij} u^j))}_{\frac{n}{2s-t}}\\
\aleq &\vrac{u}_\infty\ \vrac{\laps{t+2(s-t)} (\scutA_{2K+k} \lapms{t} (\cut_K g)))}_{\frac{n}{2s-t}}\\
&+\vrac{\laps{2(s-t)}{u}}_{\frac{n}{2(s-t)}}\ \vrac{\laps{t} (\scutA_{2K+k} \lapms{t} (\cut_K g)))}_{\frac{n}{t}}\\
&+\vrac{H_{2(s-t)}(u,\laps{t} (\scutA_{2K+k} \lapms{t} (\cut_K g)))}_{\frac{n}{2s-t}}\\
\aleq & (\vrac{\laps{t} u}_{p_t} + \vrac{u}_\infty)\ 2^{-(K+k)\sigma}\ \vrac{g}_{p_t}.
\end{align*}
In the last step we used estimates on the three-term-commutator $H$, Theorem \ref{th:bicomest}, and Sobolev inequality.

The $I$ case remains, and setting $\varphi := \scut_{2K} \lapms{t} (\cut_K g)$,
\[
 \vrac{\laps{t} \varphi}_{p_t} \aleq 1.
\]
Indeed, this again follows from Theorem~\ref{th:bicomest} and the following estimate which works for any $q \in (1,p_t)$ such that $ \frac{nq}{n-tq} \in [p_t,\infty)$
\[
 \vrac{\laps{t}\scut_{2K} \lapms{t} (\cut_K g)}_{p_t} \aleq 2^{2K(\frac{n}{p_t} - \frac{n}{q})} \ \vrac{\lapms{t} (\cut_K g)}_{\frac{nq}{n-tq}} \aleq 2^{(2K-K)(\frac{n}{p_t} - \frac{n}{q})} \aleq 1.
\]
Then $|I| \leq |I_1| + |I_2| + |I_3|$, with
\begin{align*}
   I_1:=&\omega_{ij} \int \laps{t} (\varphi  u^j)\  T_{B_L,t} u^i,\\
   I_2:=& \omega_{ij} \int \varphi  \laps{t} u^j\  T_{B_L,t} u^i,\\
   I_3:=& \omega_{ij} \int \laps{2(s-t)} H_t(\varphi,u)\  T_{B_L,2s-t} u^i.
\end{align*}
For term $I_3$, if $(s-t)$ is small enough, we can apply the localized version of Theorem~\ref{th:bicomest}, as well as Proposition~\ref{pr:lapmsTest}, and then Theorem~\ref{th:sobolev:new} (here we need to assume that $L$ is a multiple of $K$, say $L = 10K$)
\[
 |I_3| \aleq \vrac{\laps{2(s-t)} H_t(\varphi,u)}_{{\frac{n}{2s-t}}} \vrac{T_{B_L,2s-t} u}_{\frac{n}{n-2s+t}}\ \aleq [u]_{B_L}^p + \sum_{l=1}^\infty 2^{-\sigma (L+l)} [u]^{p-1}_{L+l}.
\]
Now we take care of $I_1$, employing \eqref{eq:dualargTBt},
\[
 I_1 = \int \limits_{B_{L}} \int \limits_{B_{L}} \frac{|u(x)-u(y)|^{p-2} (u^i(x)-u^i(y)) \omega_{ij} (\varphi(x) u^j(x) - \varphi(y) u^j(y))}{|x-y|^{n+\alpha p}} dx dy.
\]
We use the Euler-Lagrange system \eqref{eq:ourEL}, also using that if $L\geq 10K$, the support of $\supp \varphi \in B_{2K}$ is rather small,
\begin{align*}
 |I_1| &\leq \int \limits_{\Omega \backslash B_{L}} \int \limits_{\Omega} \frac{|u(x)-u(y)|^{p-1} |\varphi(x) u^j(x) - \varphi(y) u^j(y)|}{|x-y|^{n+\alpha p}} dx\ dy\\
  &\leq \int \limits_{\Omega \backslash B_{L}} \int \limits_{B_{2K}} \frac{|u(x)-u(y)|^{p-1} |\varphi(x) u^j(x)|}{|x-y|^{n+\alpha p}}\ dx\ dy\\
  &\aleq \vrac{u}_{\infty,\Omega} \int \limits_{\R^n \backslash B_{L}} \int \limits_{B_{2K}} \frac{|u(x)-u(y)|^{p-1} |\varphi(x)|}{|x-y|^{n+\alpha p}}\ dx\ dy\\
  &\aleq [\varphi]_{s,p,B_1}\ \sum_{l=1}^\infty 2^{-\sigma (L+l)} [u]_{L+l}^{p-1}.
\end{align*}
In the last step we used Proposition~\ref{pr:uxmuypm1varphi} and that w.l.o.g $\vrac{u}_\infty \aleq 1$ on $\R^n$.

It remains to treat
\begin{align*}
 I_2 &= \omega_{ij} \int \varphi  \laps{t} u^j\  T_{B_L,t} u^i\\
 &= \omega_{ij} \int \limits_{B_L}\int \limits_{B_L} \frac{|u(x)-u(y)|^{p-2} (u^i(x)-u^i(y))\omega_{ij} (\lapms{t} (\varphi  \laps{t} u^j)(x)-\lapms{t} (\varphi  \laps{t} u^j)(y))}{|x-y|^{n+sp}}\ dx\ dy.
\end{align*}
Now we insert the following zero, since $\omega$ is antisymmetric
\[
 (u^i(x)-u^i(y))\omega_{ij} (u^j(x)-u^j(y)) \equiv 0
\]
and add
\[
 0 = -\frac{1}{2} \omega_{ij} \int \limits_{B_L}\int \limits_{B_L} \frac{|u(x)-u(y)|^{p-2} (u^i(x)-u^i(y))\omega_{ij} (u^j(x)-u^j(y)) (\varphi(x)+\varphi(y))}{|x-y|^{n+sp}}\ dx\ dy.
\]
Now $I_2$ falls under the realm of Lemma~\ref{la:thcommis:T2} and this concludes the proof.
\end{proof}

\section{Compensation effects for commutator-like expressions}\label{s:commis}

\subsection{Preliminary estimates}
Many arguments in the following proofs are based on the following case study. We used this kind of argument in \cite[Chapter 3]{Sfracenergy} to obtain estimates for $H_s$ as in Theorem~\ref{th:bicomest}.
\begin{proposition}\label{pr:c:spacedecomp}
For almost every $x, y,z \in \R^n$, we have three cases
\begin{itemize}
\item[Case 1:]  $\abs{x-y} \leq \frac{1}{2} \abs{x-z}$ or $\abs{x-y} \leq \frac{1}{2} \abs{y-z}$,
\item[Case 2:] $2\abs{x-y} \geq \max \{\abs{x-z},\abs{x-z}\}$ and $\abs{x-z} \leq \abs{y-z}$,
\item[Case 3:] $2\abs{x-y} \geq \max \{\abs{x-z},\abs{x-z}\}$ and $\abs{x-z} > \abs{y-z}$,
\end{itemize}
and for arbitrary $\beta \in (0,n)$, $\eps \in (0,1]$:

In Case 1, $|x -z| \approx |y-z|$, and
\[
 |\abs{x-z}^{\beta - n} -\abs{y-z}^{\beta - n}| \aleq |x-y|^{\eps} \min \{\abs{x-z}^{\beta -\eps- n} ,\ \abs{y-z}^{\beta -\eps- n}\}.
\]
In Case 2,
\[
 |\abs{x-z}^{\beta - n} -\abs{y-z}^{\beta - n}| \aleq |x-y|^{\eps} \abs{x-z}^{\beta -\eps- n}.
\]
In Case 3,
\[
 |\abs{x-z}^{\beta - n} -\abs{y-z}^{\beta - n}| \aleq |x-y|^{\eps} \abs{y-z}^{\beta -\eps- n}.
\]
\end{proposition}
From Proposition~\ref{pr:c:spacedecomp} and the definition of Riesz potentials, \eqref{eq:rieszpotential}, we have the following $\beta$-H\"older-continuity estimates for $\beta \in (0,\alpha)$  
\begin{proposition}\label{pr:c:incrsimplest}
For any $\alpha \in (0,1)$, $\beta \in (0,\alpha)$, for almost every $y,z \in \R^n$ and for any $f = \lapms{\alpha} F$,
\[
 \abs{f(x)-f(y)} \leq C_{\alpha-\beta}\ \abs{x-y}^{\beta}\ \brac{\lapms{\alpha-\beta}\abs{F}(x) + \lapms{\alpha-\beta}\abs{F}(y)}.
\]
\end{proposition}

  From Proposition \ref{pr:c:incrsimplest}, we deduce
  \begin{proposition}\label{pr:c:2ndorderdiffincest}
  Let $\beta \in (0,1)$, $\alpha \in (0,1)$ and $\eps \in (0,1-\alpha)$ such that $\eps  < \min \{1-\alpha, \beta-\frac{\alpha}{2}\}$. Then,
  \begin{align*}
    &\abs{f(x)+f(y)-2f(z)} \abs{\abs{x-z}^{\beta - n} -\abs{y-z}^{\beta - n}}\\
  \aleq & 
  \brac{\lapms{\beta-{\frac{\alpha}{2}}}|\laps{\beta} f|(y) + \lapms{\beta-{\frac{\alpha}{2}}}|\laps{\beta} f|(x) + \lapms{\beta-{\frac{\alpha}{2}}}|\laps{\beta} f|(z)}\ \abs{x-y}^{\alpha + \eps}\ k_{\beta-\frac{\alpha}{2}-\eps,\beta}(x,y,z),
  \end{align*}
  where $k_{s,\gamma}$ has the form, 
  \begin{align}
  k_{s,\gamma}(x,y,z) := &\quad  \min\{\abs{y-z}^{s - n}, \abs{x-z}^{s - n}\} \label{eq:c:kityp1}\\
  &+ \brac{\frac{\abs{y-z}}{\abs{x-y}}}^{\gamma-s}\ \abs{y-z}^{s - n} \chi_{\{\abs{y-z} < 2\abs{x-y}\}\label{eq:c:kityp2}}\\
  &+ \brac{\frac{\abs{x-z}}{\abs{x-y}}}^{\gamma-s} \abs{x-z}^{s - n} \chi_{\{\abs{x-z} < 2\abs{x-y}\}\label{eq:c:kityp3}}.
  \end{align}
  \end{proposition}
  \begin{proof}
  Let
  \[
  F := \laps{\beta} f.
  \]

  We have the following simple estimate
  \[
  \abs{f(x)+f(y)-2f(z)} \leq \begin{cases}
			      \abs{f(x)-f(z)}+\abs{f(y)-f(z)},\\
				\abs{f(x)-f(y)}+2\abs{f(y)-f(z)},\\
				\abs{f(y)-f(x)} + 2\abs{f(x)-f(z)}.
			      \end{cases}
  \]
  In view of Proposition \ref{pr:c:incrsimplest}, this implies that for $\frac{\alpha}{2} \in (0,\beta)$ we have three options \eqref{eq:c:2dincchoice:1}, \eqref{eq:c:2dincchoice:2}, \eqref{eq:c:2dincchoice:3} to estimate
  \[
  \abs{f(x)+f(y)-2f(z)}:
  \]
  Firstly,
  \begin{equation}\label{eq:c:2dincchoice:1}
  \abs{x-z}^{\frac{\alpha}{2}}\ \brac{\lapms{\beta-\frac{\alpha}{2}}\abs{F}(x) + \lapms{\beta-\frac{\alpha}{2}}\abs{F}(z)} + \abs{y-z}^{\frac{\alpha}{2}}\ \brac{\lapms{\beta-\frac{\alpha}{2}}\abs{F}(y) + \lapms{\beta-\frac{\alpha}{2}}\abs{F}(z)} ,
  \end{equation}
  secondly,
  \begin{equation}\label{eq:c:2dincchoice:2}
  \abs{x-y}^{\frac{\alpha}{2}}\ \brac{\lapms{\beta-\frac{\alpha}{2}}\abs{F}(y) + \lapms{\beta-\frac{\alpha}{2}}\abs{F}(x)}
				+\abs{y-z}^{\frac{\alpha}{2}}\ \brac{\lapms{\beta-\frac{\alpha}{2}}\abs{F}(y) + \lapms{\beta-\frac{\alpha}{2}}\abs{F}(z)},
  \end{equation}
  or thirdly
  \begin{equation}\label{eq:c:2dincchoice:3}
  \abs{x-y}^{\frac{\alpha}{2}}\ \brac{\lapms{\beta-\frac{\alpha}{2}}\abs{F}(y) + \lapms{\beta-\frac{\alpha}{2}}\abs{F}(x)}
				+ \abs{x-z}^{\frac{\alpha}{2}}\ \brac{\lapms{\beta-\frac{\alpha}{2}}\abs{F}(x) + \lapms{\beta-\frac{\alpha}{2}}\abs{F}(z)}.
  \end{equation}
We now consider the cases of Proposition~\ref{pr:c:spacedecomp}:
  \begin{itemize}
\item[Case 1:]  $\abs{x-y} \leq \frac{1}{2} \abs{x-z}$ or $\abs{x-y} \leq \frac{1}{2} \abs{y-z}$,
\item[Case 2:] $2\abs{x-y} \geq \max \{\abs{x-z},\abs{x-z}\}$ and $\abs{x-z} \leq \abs{y-z}$,
\item[Case 3:] $2\abs{x-y} \geq \max \{\abs{x-z},\abs{x-z}\}$ and $\abs{x-z} > \abs{y-z}$,
\end{itemize}
In {\bf Case 1}, since then $\abs{x-z} \aeq \abs{y-z}$, we have for $\gamma_1,\gamma_2 \in [0,1]$,
  \begin{align*}
  &\abs{f(x)+f(y)-2f(z)} \abs{\abs{x-z}^{\beta - n} -\abs{y-z}^{\beta - n}}\\
  \overset{\eqref{eq:c:2dincchoice:2}}{\aleq}&
    \abs{x-y}^{{\frac{\alpha}{2}}}\ \brac{\lapms{\beta-{\frac{\alpha}{2}}}\abs{F}(y) + \lapms{\beta-{\frac{\alpha}{2}}}\abs{F}(x)}\ \abs{\abs{x-z}^{\beta - n} -\abs{y-z}^{\beta - n}}\\
    &+\abs{y-z}^{{\frac{\alpha}{2}}}\ \brac{\lapms{\beta-{\frac{\alpha}{2}}}\abs{F}(y) + \lapms{\beta-{\frac{\alpha}{2}}}\abs{F}(z)}\ \abs{\abs{x-z}^{\beta - n} -\abs{y-z}^{\beta - n}}  \\
    \aleq&  
    \abs{x-y}^{{\frac{\alpha}{2}}}\ \brac{\lapms{\beta-{\frac{\alpha}{2}}}\abs{F}(y) + \lapms{\beta-{\frac{\alpha}{2}}}\abs{F}(x)}\ \abs{y-z}^{\beta - n-\gamma_1} \abs{x-y}^{\gamma_1}\\
    &+\abs{y-z}^{{\frac{\alpha}{2}}}\ \brac{\lapms{\beta-{\frac{\alpha}{2}}}\abs{F}(y) + \lapms{\beta-{\frac{\alpha}{2}}}\abs{F}(z)}\ \abs{y-z}^{\beta - n-\gamma_2} \abs{x-y}^{\gamma_2} \\
    =& \brac{\lapms{\beta-{\frac{\alpha}{2}}}\abs{F}(y) + \lapms{\beta-{\frac{\alpha}{2}}}\abs{F}(x)}\ \abs{y-z}^{\beta - n-\gamma_1} \abs{x-y}^{\gamma_1+{\frac{\alpha}{2}}}\\
    &+ \brac{\lapms{\beta-{\frac{\alpha}{2}}}\abs{F}(y) + \lapms{\beta-{\frac{\alpha}{2}}}\abs{F}(z)}\ \abs{y-z}^{\beta - n-\gamma_2+{\frac{\alpha}{2}}} \abs{x-y}^{\gamma_2} 
  \end{align*}
Now we choose $\gamma_1 := \frac{\alpha}{2} + \eps$, $\gamma_2 = \alpha + \eps$, which is admissible by the conditions on $\eps$, and $\beta-\frac{\alpha}{2}-\eps > 0$.
  \begin{align*}
  &\abs{f(x)+f(y)-2f(z)} \abs{\abs{x-z}^{\beta - n} -\abs{y-z}^{\beta - n}}\\
    \aleq & \brac{\lapms{\beta-{\frac{\alpha}{2}}}\abs{F}(y) + \lapms{\beta-{\frac{\alpha}{2}}}\abs{F}(x) + \lapms{\beta-{\frac{\alpha}{2}}}\abs{F}(z)}\ \abs{y-z}^{\beta-\frac{\alpha}{2}-\eps-n}\ \abs{x-y}^{\alpha + \eps} \\
  \end{align*}
  Thus, in this case the kernel is of the form \eqref{eq:c:kityp1}.

  Next we have in {\bf Case 2}, for any $\gamma_1,\gamma_2 > 0$, later choosing $\gamma_1 := \frac{\alpha}{2} + \eps$, and $\gamma_2 := \alpha + \eps$,
  \begin{align*}
  &\abs{f(x)+f(y)-2f(z)} \abs{\abs{x-z}^{\beta - n} -\abs{y-z}^{\beta - n}}\\
  \overset{\eqref{eq:c:2dincchoice:2}}{\aleq}& 
  \abs{x-y}^{\frac{\alpha}{2}}\ \brac{\lapms{\beta-\frac{\alpha}{2}}\abs{F}(y) + \lapms{\beta-\frac{\alpha}{2}}\abs{F}(x)}\ \abs{y-z}^{\beta - n}\\
&+\abs{y-z}^{\frac{\alpha}{2}}\ \brac{\lapms{\beta-\frac{\alpha}{2}}\abs{F}(y) + \lapms{\beta-\frac{\alpha}{2}}\abs{F}(z)}\ \abs{y-z}^{\beta - n}\\
  =&  \brac{\lapms{\beta-\frac{\alpha}{2}}\abs{F}(y) + \lapms{\beta-\frac{\alpha}{2}}\abs{F}(x)}\ \abs{x-y}^{\frac{\alpha}{2} + \gamma_1} \abs{y-z}^{\beta -\gamma_1 - n}\ \brac{\frac{\abs{y-z}}{\abs{x-y}}}^{\gamma_1}\\
  &+\brac{\lapms{\beta-\frac{\alpha}{2}}\abs{F}(y) + \lapms{\beta-\frac{\alpha}{2}}\abs{F}(z)}\ \abs{x-y}^{\gamma_2} \abs{y-z}^{\beta - n+\frac{\alpha}{2}-\gamma_2}\ \brac{\frac{\abs{y-z}}{\abs{x-y}}}^{\gamma_1+(\gamma_2-\gamma_1)}\\
  \overset{\gamma_1 < \gamma_2}{\aleq}&  \brac{\lapms{\beta-\frac{\alpha}{2}}\abs{F}(y) + \lapms{\beta-\frac{\alpha}{2}}\abs{F}(x)}\ \abs{x-y}^{\frac{\alpha}{2} + \gamma_1} \abs{y-z}^{\beta -\gamma_1 - n}\ \brac{\frac{\abs{y-z}}{\abs{x-y}}}^{\gamma_1}\\
  &+\brac{\lapms{\beta-\frac{\alpha}{2}}\abs{F}(y) + \lapms{\beta-\frac{\alpha}{2}}\abs{F}(z)}\ \abs{x-y}^{\gamma_2} \abs{y-z}^{\beta - n+\frac{\alpha}{2}-\gamma_2}\ \brac{\frac{\abs{y-z}}{\abs{x-y}}}^{\gamma_1},
  \end{align*}
Since we are in Case 2, the kernel can be written as in \eqref{eq:c:kityp2}.
By an analogous argument from Case 3 we obtain an estimate with \eqref{eq:c:kityp3}
\end{proof}

\begin{proposition} \label{pr:c:integrguywithhxi}
Let $F,G,H : \R^n \to \R_+$, $\alpha \in (0,n)$, $s,\beta \in (0,1)$, $s+\alpha< \beta$, and consider
\[
 I := \int \limits_{\R^n} \int \limits_{\R^n} \int \limits_{\R^n} \brac{F(x) + F(y)}\ \brac{G(z)+G(x)+G(y)}\ \abs{x-y}^{\alpha-n}\ H(z)\ k_{s,\beta}(x,y,z)\ dx\ dy\ dz,
\]
where $k_{s}(x,y,z)$ is of the form \eqref{eq:c:kityp1}, \eqref{eq:c:kityp2}, or \eqref{eq:c:kityp3}. Then
\[
 I \leq \int \limits_{\R^n} G\ H\ \lapms{s+\alpha} F + \int \limits_{\R^n} F\ G\  \lapms{\alpha+s} H + \int \limits_{\R^n} F\ \lapms{\alpha} G\ \lapms{s} H + \int \limits_{\R^n} G\ \lapms{\alpha} F\ \lapms{s} H.
\]
\end{proposition}
\begin{proof}
We are going to show that
\begin{align*}
  I \leq &\int \limits_{\R^n} \lapms{\alpha} F\ \lapms{s}\brac{G H} + \int \limits_{\R^n} \lapms{\alpha}(FG)\ \lapms{s} H + \int \limits_{\R^n} F\ \lapms{\alpha} G\ \lapms{s} H\\
  &+ \int \limits_{\R^n} G\ \lapms{\alpha} F\ \lapms{s} H + \int \limits_{\R^n} F\ \lapms{s+\alpha} (GH) + \int \limits_{\R^n} FG\ \lapms{s+\alpha} H,
\end{align*}
which, by integration by parts, simplifies to the claim.

We have to consider only products of the following form, the other cases follow from symmetric considerations.
\begin{align}
 &F(x)\ G(z)\ H(z),\label{eq:c:integwithhxi:yxixi}\\
&F(x)\ G(x)\ H(z),\label{eq:c:integwithhxi:yyxi}\\
&F(y)\ G(x)\ H(z).\label{eq:c:integwithhxi:zyxi}
\end{align}
\underline{In the case of \eqref{eq:c:kityp1}, \eqref{eq:c:kityp2}}, where we have
\[
 k_{s,\beta}(x,y,z) \aleq \abs{y-z}^{s - n},
\]
we have for \eqref{eq:c:integwithhxi:yxixi},
\begin{align*}
 &\int \limits_{\R^n} \int \limits_{\R^n} \int \limits_{\R^n} F(x)\ G(z)\ \abs{x-y}^{\alpha-n}\ H(z)\ k_{s,\beta}(x,y,z)\ dx\ dy\ dz\\
&\aleq \int \limits_{\R^n} \int \limits_{\R^n} \int \limits_{\R^n} F(x)\ \abs{x-y}^{\alpha-n}\ dy\ H(z)\ G(z)\ \abs{y-z}^{s-n}\ dx\ dy\ dz\\
&\overset{\eqref{eq:rieszpotential}}{\aeq} \int \limits_{\R^n} \int \limits_{\R^n} \lapms{\alpha}F(y)\ G(z)\ H(z)\ \abs{y-z}^{s-n}\ dx\ dz\\
&\overset{\eqref{eq:rieszpotential}}{\aeq} \int \limits_{\R^n} \lapms{\alpha}F(y)\ \lapms{s}\brac{G H}(z)\ dz.
\end{align*}
Similarly, for \eqref{eq:c:integwithhxi:yyxi},
\begin{align*}
 &\int \limits_{\R^n} \int \limits_{\R^n} \int \limits_{\R^n} F(x)\ G(x)\ \abs{x-y}^{\alpha-n}\ H(z)\ k_{s,\beta}(x,y,z)\ dx\ dy\ dz\\
&\aleq \int \limits_{\R^n} \int \limits_{\R^n} \int \limits_{\R^n} F(x)\ G(x)\ \abs{x-y}^{\alpha-n}\ dx\ H(z)\ \abs{y-z}^{s-n}\ dy\ dz\\
&\overset{\eqref{eq:rieszpotential}}{\aeq} \int \limits_{\R^n} \int \limits_{\R^n} \lapms{\alpha}(FG)(y)\ H(z)\ \abs{y-z}^{s-n}\ dy\ dz\\
&\overset{\eqref{eq:rieszpotential}}{\aeq} \int \limits_{\R^n} \lapms{\alpha}(FG)(y)\ \lapms{s}H(y)\ dy.
\end{align*}
For \eqref{eq:c:integwithhxi:zyxi},
\begin{align*}
 &\int \limits_{\R^n} \int \limits_{\R^n} \int \limits_{\R^n} F(y)\ G(x)\ \abs{x-y}^{\alpha-n}\ H(z)\ k_{s,\beta}(x,y,z)\ dx\ dy\ dz\\
\aleq& \int \limits_{\R^n} F(y)\ \int \limits_{\R^n} \int \limits_{\R^n} G(x)\ \abs{x-y}^{\alpha-n}\ dx\ H(z)\ \abs{y-z}^{s-n}\ dy\ dz\\
\overset{\eqref{eq:rieszpotential}}{\aeq}& \int \limits_{\R^n} F(y) \int \limits_{\R^n} \lapms{\alpha}G(y)\ H(z)\ \abs{y-z}^{s-n}\ dy\ dz\\
\overset{\eqref{eq:rieszpotential}}{\aeq}& \int \limits_{\R^n} F(y)\ \lapms{\alpha}G(y)\ \lapms{s}H(y)\ dz.
\end{align*}
\underline{In the case of \eqref{eq:c:kityp3}}, that is
\[
 k_s(y,x,z) = \brac{\frac{\abs{x-z}}{\abs{x-y}}}^{\beta-s} \abs{x-z}^{s - n} \chi_{\{\abs{x-z} < 2\abs{x-y}\}},
\]
we have for \eqref{eq:c:integwithhxi:yxixi},
\begin{align*}
 &\int \limits_{\R^n} \int \limits_{\R^n} \int \limits_{\R^n} F(x)\ G(z)\ \abs{x-y}^{\alpha-n}\ H(z)\ k_{s,\beta}(x,y,z)\ dx\ dy\ dz\\
\aleq& \int \limits_{\R^n} F(x) \int \limits_{\R^n} \int \limits_{\{\abs{x-y} \ageq \abs{x-z}\}} \abs{x-y}^{s+\alpha-\beta-n}\ dy\quad H(z)\ G(z)\ \abs{x-z}^{\beta-n}\ dz\ dx\\
\overset{s+\alpha < \beta}{\aeq}& \int \limits_{\R^n} F(x) \int \limits_{\R^n} \abs{x-z}^{s+\alpha-\beta-n}\ H(z)\ G(z)\ \abs{x-z}^{\beta-n}\ dz\ dx\\
\overset{\eqref{eq:rieszpotential}}{\aeq}& \int \limits_{\R^n} F(x)\ \lapms{s+\alpha} (HG)(x)\ dx.
\end{align*}
Similarly, for \eqref{eq:c:integwithhxi:yyxi},
\begin{align*}
&\int \limits_{\R^n} \int \limits_{\R^n} \int \limits_{\R^n} F(x)\ G(x)\ \abs{x-y}^{\alpha-n}\ H(z)\ k_{s,\beta}(x,y,z)\ dx\ dy\ dz\\
\aleq& \int \limits_{\R^n} F(x)\ G(x) \int \limits_{\R^n} \int \limits_{\{\abs{x-y} \ageq \abs{x-z}\}} \abs{x-y}^{s+\alpha-\beta-n}\ dy\quad H(z)\ \abs{x-z}^{\beta-n}\ dz\ dx\\
\overset{s+\alpha < \beta}{\aeq}&
\int \limits_{\R^n} F(x)\ G(x) \int \limits_{\R^n} \abs{x-z}^{s+\alpha-n}\ H(z)\ dz\ dx\\
\overset{\eqref{eq:rieszpotential}}{\aeq}& \int \limits_{\R^n} F(x)\ G(x)\ \lapms{s+\alpha} H(x)\ dy.
\end{align*}
Lastly, for \eqref{eq:c:integwithhxi:zyxi},
\begin{align*}
&\int \limits_{\R^n} \int \limits_{\R^n} \int \limits_{\R^n} F(y)\ G(x)\ \abs{x-y}^{\alpha-n}\ H(z)\ k_{s,\beta}(x,y,z)\ dx\ dy\ dz\\
\aleq& \int \limits_{\R^n} G(x)\ \int \limits_{\R^n} F(y)\ \abs{x-y}^{\alpha-n}\ \int \limits_{\R^n} H(z)\ \abs{x-z}^{s-n}\ dz\ dy\ dx\\
\overset{\eqref{eq:rieszpotential}}{\aeq}& \int \limits_{\R^n} G(x)\ \lapms{\alpha} F(x)\ \lapms{s}H(x)\ dx.
\end{align*}
This concludes the proof of Proposition~\ref{pr:c:integrguywithhxi}.
\end{proof}

\subsection{The Compensation Estimates: Proof of Theorem~\ref{th:threecomms}}
\begin{lemma}\label{la:thcommis:T1}
Fix $s \in (0,1)$. For all $t < s$ large enough, let
\begin{equation}\label{eq:thcommis:T1:new}
T_1(z) := \int \limits_{B_{\rho}}\int \limits_{B_{\rho}} \frac{|f(x)-f(y)|^{p_s-1}\ |\Gamma(x,y,z)|}{|x-y|^{n+sp_s}}\ dx\ dy,
\end{equation}
where
\[
 \Gamma(x,y,z) = |g(x) + g(y)-2g(z)|\ ||x-z|^{t-n} -|y-z|^{t-n}|
\]
Then we have for any $L \in \N$,
\[
\vrac{T_1}_{p_t'} \aleq [f]_{B_{2^L \rho},s,p_s}^{{p_s-1}}\ [g]_{B_{2^L \rho},s,p_s} + \sum_{k=1}^\infty 2^{-\sigma(L+l)} [f]_{B_{2^{L+l} \rho},s,p_s}^{p_s-1}\ [g]_{B_{2^{L+l} \rho},s,p_s}.
\]
\end{lemma}
\begin{proof}
Let $F := |\laps{t} f|$, $G := |\laps{t} f|$ both of which by Theorem~\ref{th:sobolev:new} satisfy
\begin{equation}\label{eq:t1:FvsfGvsg}
\vrac{F}_{p_t} \aleq [f]_{s,p_s,\R^n}, \quad \vrac{G}_{p_t} \aleq [g]_{s,p_s,\R^n}
\end{equation}
By Proposition~\ref{pr:c:incrsimplest}, for any small $\delta > 0$,
\[
 \abs{f(x)-f(y)}^{p_s-1} \aleq \abs{x-y}^{(t-\delta)(p_s-1)}\ \brac{(\lapms{\delta}F)^{p_s-1}(x) + (\lapms{\delta}F)^{p_s-1}(y)},
\]
 and Proposition~\ref{pr:c:2ndorderdiffincest}, for $\eps < t-\frac{s}{2}$,
\[
  \Gamma(x,y,z) \aleq \brac{\lapms{t-{\frac{s}{2}}}G(y) + \lapms{t-{\frac{s}{2}}}G(x) + \lapms{t-{\frac{s}{2}}}G(z)}\ \abs{x-y}^{s + \eps}\ k_{t-\frac{s}{2}-\eps,t}(x,y,z).
\]
Consequently, for some $\varphi \in C_0^\infty(\R^n)$, $\vrac{\varphi}_{{p_t}} \leq 1$
\[
 \vrac{T_1}_{p_t'} \aleq \int \limits_{\R^n} \int \limits_{\R^n} \int \limits_{\R^n} \frac{\Theta(x,y,z)}{|x-y|^{n+(p_s-1)(s-t+\delta)-\eps}}\ dx\ dy\ dz,
\]
where $\Theta(x,y,z)$ is composed by the the following terms, using also symmetry of $x$ and $y$,
\begin{align}
 \label{eq:thcommis:T1est:fckxx}
 k_{t-\frac{s}{2}-\eps,t}(x,y,z)\ |\varphi|(z)\ \lapms{t-{\frac{s}{2}}}G(x)\ \cut_{B_\rho}(x)(\lapms{\delta}F)^{p_s-1}(x)\\
 \label{eq:thcommis:T1est:fckxy}
 k_{t-\frac{s}{2}-\eps,t}(x,y,z)\ |\varphi|(z)\ \lapms{t-{\frac{s}{2}}}G(x)\ \cut_{B_\rho}(y)(\lapms{\delta}F)^{p_s-1}(y)\\
 \label{eq:thcommis:T1est:fckzx}
 k_{t-\frac{s}{2}-\eps,t}(x,y,z)\ |\varphi|(z)\ \lapms{t-{\frac{s}{2}}}G(z)\ \cut_{B_\rho}(x)(\lapms{\delta}F)^{p_s-1}(x)
\end{align}
We can choose $\delta$ small enough and $t$ close enough to $s$ so that an admissible $\eps > 0$ guarantees that
\[
 \alpha := \eps-(s-t+\delta)(p_s-1) > 0.
\]
Now the conditions for Proposition~\ref{pr:c:integrguywithhxi} are satisfied, since always
\[
 t - \frac{s}{2} - \eps + \alpha < t.
\]
Let 
\[
 \tilde{G} := \lapms{t-{\frac{s}{2}}}G \in L^{2\frac{n}{s}}
\]
\[
 \tilde{F} := \cut_{B_\rho}(\lapms{\delta}F)^{p_s-1} \in L^{\frac{sn}{(t-\delta)(n-s)}} \subset L^1_{loc}
\]
We now apply Proposition~\ref{pr:c:integrguywithhxi},
\begin{align*}
\leq &\int \limits_{\R^n} \tilde{G}\ \varphi\ \lapms{t-\frac{s}{2}-\eps+\alpha} \tilde{F} + \int \limits_{\R^n} \tilde{F}\ \tilde{G}\  \lapms{t-\frac{s}{2}-\eps+\alpha} \varphi\\
&+ \int \limits_{\R^n} \tilde{F}\ \lapms{\alpha} \tilde{G}\ \lapms{t-\frac{s}{2}-\eps} \varphi + \int \limits_{\R^n} \tilde{G}\ \lapms{\alpha} \tilde{F}\ \lapms{t-\frac{s}{2}-\eps} \varphi.
\end{align*}
First of all, these integrals make sense: Possibly using partial integration,
\[
 \int (\lapms{\gamma} f)\ g = \int f\ \lapms{\gamma}g,
\]
one checks that by H\"older and classical Sobolev inequality, Theorem~\ref{th:sobolev:classic}, and then \eqref{eq:t1:FvsfGvsg},
\[
 \int T_1\ \varphi \aleq \vrac{F}_{p_t}^{p_s-1}\ \vrac{G}_{p_t}\ \ \vrac{\varphi}_{p_t} \aleq [f]^{p_s-1}_{p_s,s,\R^n}\ [g]_{p_s,s,\R^n}.
\]
To localize this argument note that $\tilde{F}$ has a cutoff function $\cut_{B_\rho}$. Then we can apply Proposition~\ref{pr:tripleintegral}, and several times Proposition~\ref{pr:classiclocalsobolev}, and finally Lemma~\ref{la:locsob}, to obtain the claim.
\end{proof}

\begin{lemma}\label{la:thcommis:T2}.
\begin{equation}\label{eq:thcommis:T2:new}
T_2 := \int \limits_{B_{\rho}}\int \limits_{B_{\rho}} \frac{|f(x)-f(y)|^{p_s-1}\ |\Gamma(x,y)|}{|x-y|^{n+sp}}\ dx\ dy,
\end{equation}
where
\[
 \Gamma(x,y) = \lapms{t} (g \laps{t} h)(x)-\lapms{t} (g \laps{t} h)(y) - \frac{1}{2} (h(x)-h(y)) (g(x)+g(y))
\]
Then we have
\[ T_2 \aleq \vrac{\laps{t} g}_{p_t}\ [f]_{B_{2^L \rho},s,p_s}^{p_s-1}\ [h]_{B_{2^L \rho},s,p_s} + \vrac{\laps{t} g}_{p_t} \sum_{k=1}^\infty 2^{-\sigma(L+l)} [f]_{B_{2^{L+l} \rho},s,p_s}^{p_s-1}\ [h]_{B_{2^{L+l} \rho},s,p_s}
\]
\end{lemma}
\begin{proof}
Let $F := |\laps{t} f|$, $G := |\laps{t} g|$, $H := |\laps{t} h|$.

To prove \eqref{eq:thcommis:T2:new}, first we observe,
\begin{align*}
 \Gamma(x,y) &= \lapms{t} (g H)(x)-\lapms{t} (g H)(y) - \frac{1}{2} (\lapms{t}H(x)-\lapms{t}H(y)) (g(x)+g(y))\\
 &= \int \limits_{\R^n} (|x-z|^{t-n} -|y-z|^{t-n})\ g(z)\ H(z)\ dz\\
 &\quad - \frac{1}{2} \int \limits_{\R^n} (|x-z|^{t-n}-|y-z|^{t-n})\ H(z) (g(x)+g(y))\ dz\\
 &= -\frac{1}{2} \int \limits_{\R^n} (|x-z|^{t-n} -|y-z|^{t-n})\ H(z)\ (g(x)+g(y)-2g(z))\ dz.
\end{align*}
In view of Proposition \ref{pr:c:2ndorderdiffincest}, for $t <s$ close enough to $s$, and $\eps < t-\frac{s}{2} <1$ small enough
\begin{align*}
 |\Gamma(x,y)| &\aleq  \int \limits_{\R^n} ||x-z|^{t-n} -|y-z|^{t-n}|\ |H(z)|\ |g(x)+g(y)-2g(z)|\ dz\\
 &\aleq  \int \limits_{\R^n} H(z)\ \brac{\lapms{t-{\frac{s}{2}}}G(x) + \lapms{t-{\frac{s}{2}}}G(y) + \lapms{t-{\frac{s}{2}}}G(z)}\ \abs{x-y}^{s + \eps}\ k_{t-\frac{s}{2}-\eps}(x,y,z)\ dz
\end{align*}
Before we estimate $T_2$ we also need by Proposition~\ref{pr:c:incrsimplest}, which ensures, for $\delta > 0$
\[
  \abs{f(x)-f(y)}^{p_s-1} \aleq \abs{y-z}^{(t-\delta)(p_s-1)}\ \brac{(\lapms{\delta}F)^{p_s-1}(x) + (\lapms{\delta}F)^{p_s-1}(y)}.
\]
So, all in all for $T_2$, we have to estimate
\[
 T_2 \leq \int \limits_{\R^n}\int \limits_{\R^n} \int \limits_{\R^n} \Theta (x,y,z)\  |x-y|^{-n-(s-t+\delta)(p_s-1)+\eps} dz\ dx\ dy.
\]
Here $\Theta(x,y,z)$ is composed by the the following terms, using also symmetry of $x$ and $y$,
\begin{align}
 \label{eq:thcommis:T2est:fckxx}
 k_{t-\frac{s}{2}-\eps}(x,y,z)\ H(z)\ \cut_{B_\rho}(x)\lapms{t-{\frac{s}{2}}}G(x)\ \cut_{B_\rho}(x)(\lapms{\delta}F)^{p_s-1}(x)\\
 \label{eq:thcommis:T2est:fckxy}
 k_{t-\frac{s}{2}-\eps}(x,y,z)\ H(z)\ \cut_{B_\rho}(x)\lapms{t-{\frac{s}{2}}}G(x)\ \cut_{B_\rho}(y)(\lapms{\delta}F)^{p_s-1}(y)\\
 \label{eq:thcommis:T2est:fckzx}
 k_{t-\frac{s}{2}-\eps}(x,y,z)\ H(z)\ \lapms{t-{\frac{s}{2}}}G(z)\ \cut_{B_\rho}(x)(\lapms{\delta}F)^{p_s-1}(x)
\end{align}
This is exactly the same term as in the proof of Lemma~\ref{la:thcommis:T1}, and we conclude the same way.
\end{proof}

\section{Sobolev Inequality}\label{s:sobolev}
An important ingredient in our argument is the Sobolev inequality. The classical one, which we throughout our arguments used
\begin{theorem}[Classical Sobolev inequality]\label{th:sobolev:classic}
For $0 \leq t_1 < t_2$,
\[
 \vrac{\laps{t_1} f}_{p_1,\R^n} \aleq \vrac{\laps{t_2} f}_{p_2,\R^n},
\]
or in other words
\[
 \vrac{\lapms{t_2-t_1} g}_{p_1,\R^n} \aleq \vrac{g}_{p_2,\R^n},
\]
where $p_1,p_2 \in (1,\infty)$ and
\[
 \frac{1}{p_1} = \frac{1}{p_2} - \frac{t_2-t_1}{n}.
\]
\end{theorem}
We need a better imbedding, which is a special case of the Sobolev embedding for Triebel spaces, to the best of our knowledge first proved in  \cite{Jawerth77}, see also the presentation in \cite[Theorem 2.71]{Triebel1983}. Since the proof for our special situation simplifies, for convenience of the reader, we will present the arguments in Section~\ref{sec:proofsobolev}.
\begin{theorem}[Sobolev inequality]\label{th:sobolev:new}
For any $s > t \geq 0$, $p \in (1,\frac{n}{s-t})$, setting ${p_{s,t}}^\ast = \frac{np}{n-(s-t)p}$ we have
\[
 \vrac{\laps{t} f}_{{p_{s,t}^\ast},\R^n} \aleq \brac{\int \limits_{\R^n} \int \limits_{\R^n} \frac{|f(x)-f(y)|^{p}}{|x-y|^{n+sp}}\ dz\ dy}^{\frac{1}{p}}.
\]
\end{theorem}

\begin{remark}
It is worth noting, that for $p > 2$ this Sobolev inequality in Theorem~\ref{th:sobolev:new} is better than the usual one
\[
 \brac{\int \limits_{\R^n} \int \limits_{\R^n} \frac{|f(z)-f(x)|^{{p_{s,t}^\ast}}}{|y-z|^{n+tp}}\ dz\ dy}^{\frac{1}{{p_{s,t}^\ast}}} \aleq \brac{\int \limits_{\R^n} \int \limits_{\R^n} \frac{|f(z)-f(x)|^{p}}{|y-z|^{n+sp}}\ dz\ dy}^{\frac{1}{p}}.
\]
The latter one clearly holds also for $t = s$, but the constant in Theorem~\ref{th:sobolev:new} has to blow up as $t \to s$: Writing that inequality in terms of Triebel spaces (cf. \cite{GrafakosMF,Triebel1983}) $F^{s}_{p,q}$, Theorem~\ref{th:sobolev:new} states that
\[
 \vrac{f}_{F^t_{p,2}} \aleq \vrac{f}_{F^s_{p,p}},
\]
which is true only for $t < s$, but fails for $t = s$, if $p > 2$.
\end{remark}

Another form of the above Sobolev inequality is
\begin{lemma}\label{la:slobsob}
Let $s+ \delta < n$ and $p \leq \frac{n}{\delta}$. Then for ${p_{s,t}^\ast} := \frac{np}{n-\delta p}$, and for any $f \in \Sw(\R^n)$,
\[
 \brac{\int \limits_{\R^n} \int \limits_{\R^n} \frac{|\lapms{s+\delta}f(x)-\lapms{s+\delta}f(y)|^{{p_{s,t}^\ast}}}{|x-y|^{n+s{p_{s,t}^\ast}}}\ dx\ dy}^{\frac{1}{p_{s,t}^\ast}} \aleq \vrac{f}_{L^p}.
\]
\end{lemma}
\begin{proof}
This follows from the theory of Triebel spaces, \cite[\textsection 5.2.3]{Triebel1983}, and Sobolev embedding on Triebel spaces. We also outline another, more direct proof:
\begin{align*}
 &\int \limits_{\R^n} \int \limits_{\R^n} \frac{|\lapms{s+\delta}f(x)-\lapms{s+\delta}f(y)|^{{p_{s,t}^\ast}}}{|x-y|^{n+s{p_{s,t}^\ast}}}\ dx\ dy\\
 \aleq &\int \limits_{\R^n} \int \limits_{\R^n} \int \limits_{\R^n} \frac{|\lapms{s+\delta}f(x)-\lapms{s+\delta}f(y)|^{{p_{s,t}^\ast}-1}\ ||z-y|^{t-n}- |z-x|^{t-n}|\ f(z)}{|x-y|^{n+s{p_{s,t}^\ast}}}\ dx\ dy.
\end{align*}
Now one consideres the three cases of Proposition~\ref{pr:c:spacedecomp}, and for these cases one estimates
\[
 ||z-y|^{t-n}- |z-x|^{t-n}|
\]
as in Proposition~\ref{pr:c:2ndorderdiffincest}, i.e. with the kernels $k_s$. Then one integrates just as in Proposition~\ref{pr:c:integrguywithhxi}, and uses classical Sobolev inequality to obtain the claim. We leave the details to the reader.
\end{proof}

We will also need a localized version of the Sobolev inequality from Theorem~\ref{th:sobolev:new}:
\begin{lemma}\label{la:locsob}
Given $0 < t < s < 1$, $p \in (1,\frac{n}{s-t})$ the following is true. Fix a reference ball $B_R(x_0)$, and recall Definition~\ref{def:cutoffs}. Then for any $L \in \Z$, $K \in \N$, setting ${p_{s,t}^\ast} := \frac{np}{n-(s-t)p}$
\[
 \vrac{\cut_{L} \laps{t} f}_{{p_{s,t}^\ast}} \aleq [f]_{p_s,L+K} + \sum_{k =1}^\infty  2^{-\sigma(K+k)} [f]_{p_s,L+K+k}.
\]
\end{lemma}

\subsection{Proof of Theorem~\ref{th:sobolev:new}}\label{sec:proofsobolev}
We follow the presentation in \cite[Theorem 2.71]{Triebel1983}. First we need some definions:

The Littlewood-Paley theory is a mighty tool in harmonic analysis. We are going to need only very special bits and pieces, for a more general picture we refer to \cite{GrafakosMF}, and the Triebel Monographs, e.g. \cite{Triebel1983}.
We define the Littlewood-Paley projections $P_j$, which satisfy
\[
 P_j f(x) := \int \limits_{\R^n} 2^{jn} p(2^{j} (z-x)) f(z)\ dz,
\]
where $p \in \Sw(\R^n)$, $\supp p^\vee \subset B_{2}(0) \backslash B_{1/2}(0)$. Here and henceforth $f^\wedge$ is the Fourier transoform of $f$ and $f^\vee$ the inverse of the Fourier transof For convenience, we abbreviate $f_j \equiv P_j f$. We have
\[
 \sum_{j \in \Z} f_j = f \quad \mbox{for all $f \in \Sw'$}.
\]
The support-condition in particular implies
\begin{equation}\label{eq:proj:meanvaluezero}
 \int \limits_{\R^n} p = 0 
\end{equation}
Moreover, since $p \in \Sw(\R^n)$, we have for any $s,t \in [0,\infty)$
\begin{equation}\label{eq:proj:polynomialmult}
 \sup_{x \in \R^n} |x|^s |\laps{t} p(x)| \leq C_{s,t} < \infty.
\end{equation}
Next, we will also use the following which immediately follows from $p \in \Sw(\R^n)$, for any $q > 0$.
\begin{equation}\label{eq:proj:pboundedawayfromzero} 
|p(x)| \leq \frac{C_q}{1+|x|^q}.
\end{equation}

Moreover,
\begin{proposition}\label{pr:supjfjagainstslob}
For any $p \in (1,\infty)$, $s > 0$, $t \geq 0$, we have
\[
 \sup_{j \in \Z} |\laps{t} f_j(x)|^p \aleq 2^{j(t-s)p}\ \int \limits_{\R^n}  \brac{\frac{|f(z)-f(x)|}{|x-z|^s}}^p\ \frac{dz}{|x-z|^n}
\]
for any $x \in \R^n$, and
\[
 \sup_{x \in \R^n} \sup_{j \in \Z} |\laps{t} f_j(x)| \aleq 2^{j(\frac{n}{p}+t-s)}\ \brac{\int \limits_{\R^n} \int \limits_{\R^n} \brac{\frac{|f(y)-f(z)|}{|y-z|^s}}^p\ \frac{dz\ dy}{|y-z|^n}}^{\frac{1}{p}}.
\]
\end{proposition}
\begin{proof}
For any $p \in (0,1)$, by H\"older's inequality
\begin{align*}
 &|\laps{t} f_j(x)|^p\\
 \overset{\eqref{eq:proj:meanvaluezero}}{\leq}& \brac{\int \limits_{\R^n} 2^{jn} 2^{jt} |(\laps{t} p)(2^{j} (z-x))|\ |f(z)-f(x)|\ dz }^p\\
 \leq& 2^{jtp}\ \brac{\int \limits_{\R^n} 2^{jn} |(\laps{t} p)(2^{j} (z-x))|}^{\frac{p}{p'}}\ \int \limits_{\R^n} 2^{jn} |(\laps{t} p)(2^{j} (z-x))|\ |f(z)-f(x)|^p\ dz \\
 \aeq& 2^{jtp}\ \int \limits_{\R^n} 2^{jn} |(\laps{t} p)(2^{j} (z-x))|\ |f(z)-f(x)|^p\ dz \\
 =& 2^{j(t-s)p} \int \limits_{\R^n}  |2^j(z-x)|^{n+sp}|(\laps{t} p)(2^{j} (z-x))|\ \brac{\frac{|f(z)-f(x)|}{|x-z|^s}}^p\ \frac{dz}{|x-z|^n} \\
 \overset{\eqref{eq:proj:polynomialmult}}{\aleq} &2^{j(t-s)p} \int \limits_{\R^n}  \brac{\frac{|f(z)-f(x)|}{|x-z|^s}}^p\ \frac{dz}{|x-z|^n}
\end{align*}
This settles the first claim.

As for the second claim, since $P_k P_j = 0$ for $|j-k| > 1$,
\[
 \laps{t} f_j(x) = \sum_{|k-j| \leq 1} P_j (\laps{t} f_k) (x).
\]
Consequently,
\[
\vrac{\laps{t} f_j}_{\infty} \aleq \max_{k \approx j}\vrac{\brac{2^{kn}p(2^k\cdot)} \ast (\laps{t} f_j)}_{\infty} \aleq \max_{k \approx j} \vrac{\brac{2^{kn}p(2^k\cdot)}}_{p'}\ \vrac{\laps{t} f_j}_{p}.
\]
Now the second claim follows from the first one, using that
\[
 \brac{2^{kn}p(2^k\cdot)}_{p'} \aeq 2^{k\frac{n}{p}}. 
\]
This concludes the proof of Proposition~\ref{pr:supjfjagainstslob}.
\end{proof}

Now we are ready to give the
\begin{proof}[Proof of Theorem~\ref{th:sobolev:new}]
Set 
\[
 R(x) := \brac{\int \limits_{\R^n}  \brac{\frac{|f(z)-f(x)|}{|x-z|^s}}^p\ \frac{dz}{|x-z|^n}}^{\frac{1}{p}},
\]
and
\[
 \Lambda := \vrac{R}_{p} = \brac{\int \limits_{\R^n} \int \limits_{\R^n}  \frac{|f(z)-f(x)|^p}{|x-z|^{n+sp}}\ dz\ dy}^{\frac{1}{p}},
\]
W.l.o.g. $\Lambda < \infty$. Using the projections $P_j f \equiv f_j$, we set
\[
 F := \sum_{j \in \Z}|\laps{t}f_j|.
\]
We then need to show
\begin{equation}\label{eq:lasob:goal}
 \vrac{F}_{{p_{s,t}^\ast}} \aleq \Lambda.
\end{equation}
For arbitrary $L \in \Z$ we decompose
\[
 F \aleq G_L + H_L,
\]
where
\[
 G_L := \sum_{j \leq L} |\laps{t}f_j|, \quad H_L := \sum_{j > L} |\laps{t}f_j|.
\]
By Proposition~\ref{pr:supjfjagainstslob},
\[
 G_L \leq \sum_{j \leq L}  2^{j(\frac{n}{p}+t-s)}\ \Lambda \leq 2^{L(\frac{n}{p}+t-s)}\ \Lambda,
\]
where the constants are independent of $L \in \Z$. We used here that
\[
 \frac{n}{p}+t-s = \frac{n}{{p_{s,t}^\ast}} > 0.
\]
For $H_L$, Proposition~\ref{pr:supjfjagainstslob} implies
\[
 H_L = \sum_{j > L} 2^{-j(s-t)}\ R(x) \overset{t < s}{\aeq} 2^{-L(s-t)}\ R(x).
\]
Now we estimate for $l \in \Z$, $\alpha \in (2^{l-1},2^l)$
\[
 |\{|F| > \alpha \}| \leq |\{|G_L| > 2^{l-2} \}| +  |\{|H_L| \geq \frac{\alpha}{2} \}|
\]
We apply this to $L_l = \left \lfloor \frac{l}{\frac{n}{p}-(s-t)} \right \rfloor - K$, for some large $K = K(\Lambda)$ which independently of $l$ ensures
\[
 |\{|G_{L_l}| > 2^{l-2} \}| = 0.
\]
Thus
\begin{align*}
 |\{|F| > \alpha \}| &\leq |\{|H_{L_l}| \geq \frac{\alpha}{2} \}|\\
 &\leq |\{|R| \ageq 2^{(s-t)L_l}\alpha \}|\\
 &\leq |\{|R| \ageq 2^{l \brac{\frac{n}{n-p(s-t)}}} \}|\\
 &\geq  |\{|R| \ageq \alpha^{\frac{{p_{s,t}^\ast}}{p}} \}|.
\end{align*}
Thus,
\begin{align*}
 \vrac{F}_{{p_{s,t}^\ast}}^{p_{s,t}^\ast} = &\sum_{l \in \Z} \int \limits_{(2^{l-1},2^l)} \alpha^{{p_{s,t}^\ast}}\ |\{|F| > \alpha\}| \frac{d\alpha}{\alpha}\\
 \aleq &\sum_{l \in \Z} \int \limits_{(2^{l-1},2^l)} \alpha^{{p_{s,t}^\ast}}\ |\{|R| > \alpha^{\frac{{p_{s,t}^\ast}}{p}}\}| \frac{d\alpha}{\alpha}\\
  \aeq &\int \limits_{0}^\infty \beta^{p}\ |\{|R| > \beta\}| \frac{d\beta}{\beta}\\
  = & \vrac{R}_{p} = \Lambda^p.
\end{align*}
Note that the constants depend on $\Lambda$ (via the choice of $K$). This shows \eqref{eq:lasob:goal} for all $f$ with $\Lambda \equiv \Lambda_f = 1$, and by a scaling argument Theorem~\ref{th:sobolev:new} is proven.
\end{proof}
\subsection{Proof of Lemma~\ref{la:locsob}}
We need the following estimates
\begin{proposition}\label{pr:locsob:outerguy}
Fix a reference ball $B_R(x_0)$ and recall Definition~\ref{def:cutoffs}. Let $t \in (0,s)$, ${p_{s,t}^\ast} := \frac{np_s}{n-(s-t)p_s}$, then there is $\sigma \in (0,t)$
\[
 \vrac{\cut_L \laps{t} ((1-\scut_{L+K})(f-(f)_{L+K}))}_{{p_{s,t}^\ast}} \aleq \sum_{k =1}^\infty  2^{-\sigma(K+k)} [f]_{p_s,L+K+k}.
\]
\end{proposition}
\begin{proof}
We may assume $L = 0$. First
\[
\vrac{\cut_0 \laps{t} ((1-\scut_{K})(f-(f)_{K})}_{{p_{s,t}^\ast}} \aleq \sum_{k =1}^\infty \vrac{\cut_L \laps{t} (\scutA_{K+k}(f-(f)_{K})}_{{p_{s,t}^\ast}}
\]
By Lemma~\ref{la:local},
\[
 \vrac{\cut_0 \laps{t} (\scutA_{K+k}(f-(f)_{K})}_{{p_{s,t}^\ast}} \aleq R^{\frac{n}{{p_{s,t}^\ast}}}\ (2^{K+k}R)^{-n-t}\ \vrac{\scutA_{K+k}(f-(f)_{K})}_1.
\]
Next,
\begin{align*}
 &\vrac{\scutA_{K+k}(f-(f)_{K})}_1 \\
 \aleq& \vrac{\scutA_{K+k}(f-(f)_{K+k})}_1 + \sum_{l=1}^k (2^{K+k}R)^{n} |(f)_{K+l}-(f)_{K+l-1}|\\
 \aleq&  \sum_{l=1}^k (2^{K+k}R)^{n} (2^{K+l}R)^{-n} \vrac{\cut_{K+l}(f-(f)_{K+l})}_1\\
 =&  \sum_{l=1}^k 2^{kn-ln}\ \vrac{\cut_{K+l}(f-(f)_{K+l})}_1
\end{align*}
Now we use
\begin{align*}
&\vrac{\cut_{K+l}(f-(f)_{K+l})}_1 \\
\aleq & (2^{K+l}R)^{-n} \int\int \cut_{K+l}(x)\ \cut_{K+l}(y)\  |f(y)-f(z)|\ dy\ dz\\
\aleq & (2^{K+l}R)^{-n+2\frac{n}{p_s'}} \brac{\int\int \cut_{K+l}(x)\ \cut_{K+l}(y)\  |f(y)-f(z)|^{p_s}\ dy\ dz}^{\frac{1}{p_s}}\\
\aleq & (2^{K+l}R)^{-n+2\frac{n}{p_s'}+\frac{n+s p_s}{p_s}}\ [f]_{p_s,K+l}\\
\aleq & (2^{K+l}R)^{-n+2\frac{n}{p_s'}+\frac{n+s p_s}{p_s}}\ [f]_{p_s,K+k}
\end{align*}
Plugging this all in, using the definition $p_s = \frac{n}{s}$, we arrive at
\begin{align*}
 &\vrac{\cut_0 \laps{t} ((1-\scut_{K})(f-(f)_{K})}_{{p_{s,t}^\ast}} \\
 \aleq & \sum_{k =1}^\infty  [f]_{p_s,K+k}\ 2^{-t(K+k)}\ \sum_{l=1}^k 1\\
 \aleq & \sum_{k =1}^\infty  [f]_{p_s,K+k}\ k2^{-t(K+k)}\\
 \aleq & \sum_{k =1}^\infty  2^{-\sigma(K+k)} [f]_{p_s,K+k}.
\end{align*}

\end{proof}
By the Sobolev inequality, Theorem~\ref{th:sobolev:new} and Proposition~\ref{pr:psiest} we also have
\begin{proposition}\label{pr:locsob:innerguy}
Fix a reference ball $B_R(x_0)$ and recall Definition~\ref{def:cutoffs}. Let $t \in (0,s)$, $s \in (0,1)$, ${p_{s,t}^\ast} := \frac{np_s}{n-(s-t)p_s}$, then there is $\sigma \in (0,1)$
\[
 \vrac{\laps{t} (\scut_{L}(f-(f)_{L})}_{{p_{s,t}^\ast}} \aleq [f]_{p_s,L+K} + \sum_{k =1}^\infty  2^{-\sigma(K+k)} [f]_{p_s,L+K+k}.
\]
\end{proposition}
\begin{proof}[Proof of Lemma~\ref{la:locsob}]
The claim follows from Proposition~\ref{pr:locsob:outerguy}, Proposition~\ref{pr:locsob:innerguy}, using
\[
 \vrac{\cut_{L} \laps{t} f}_{{p_{s,t}^\ast}} \leq \vrac{\laps{t} \eta_{L+K} (f-(f)_{L+K})}_{{p_{s,t}^\ast}} + \vrac{\cut_{L} \laps{t} (1-\eta_{L+K})(f-(f)_{L+K})}_{{p_{s,t}^\ast}},
\]
which is true, since $\laps{t	} const = 0$.
\end{proof}
%
%
%
%
\begin{appendix}
\section{Three-Term-Commutator Estimates}
Let for $\alpha > 0$ the three term commutator given as
\[
 H_\alpha(a,b) := \laps{\alpha}(ab) - b\laps{\alpha}a -a\laps{\alpha}b.
\]
A version similar to $H$ was first was introduced (to the best of our knowledge) in the pioneering \cite{DR1dSphere}, see Theorem~\ref{th:bicomest}. They treated these commutators with the powerful tool of Littlewood-Paley decomposition. A more elementary approach, but less efective for limit estimates in Hardy-space and BMO was introduced in \cite{Sfracenergy}. The following estimate can be deduced from both arguments, see also \cite[Lemma A.5]{BPSknot12},\cite{DSpalphSphere}.
\begin{theorem}\label{th:bicomest}
For any small $\eps \geq 0$, 
\[
 \vrac{\laps{\eps} H_\alpha(a,b)}_{p} \aleq \vrac{\laps{\alpha} a}_{p_1}\ \vrac{\laps{\alpha} b}_{p_2},
\]
where for $p \in (1,\infty)$ $p_1, p_2 \in (1,\frac{n}{\alpha}]$,
\[
 \frac{1}{p} = \frac{1}{p_1} + \frac{1}{p_2} - \frac{\alpha - \eps}{n}.
\]
If $\supp a \subset B_K$, then we have 
\[
 \vrac{\laps{\eps} H_\alpha(a,b)}_{p} \aleq \vrac{\laps{\alpha} a}_{p_1}\ \brac{\vrac{\laps{\alpha} b}_{p_2,B_{K+L}} + \sum_{k=1}^\infty 2^{-\sigma (L+k)} \vrac{\laps{\alpha} b}_{p_2,B_{K+L+k}}}.
\]
\end{theorem}

\section{Localization arguments}
We collect here some results which are related to localization
\begin{lemma}\label{la:local}
Let $s \in (-n,n)$, and if $s > 0$, and $T^s$ defined as follows. \begin{itemize} \item if $s > 0$, $T^s = \nabla^s$ or $T^s = \laps{s}$ \item if $s = 0$, $T^0 = \Rz_\alpha$, for any $\alpha \in \{1,\ldots,n\}$,
\item and if $s < 0$, $T^s = \lapms{s}$. \end{itemize}
Then, $l \geq k+1$, for any $f$,
\[
 \vrac{\cutA_l T^s[\cut_k f]}_\infty \aleq (2^k)^{-n-s} \vrac{\cut_k f}_1
\]
and
\[
 \vrac{\cut_k T^s[\cutA_l f]}_\infty \aleq (2^l)^{-n-s} \vrac{\cutA_l f}_1
\]
\end{lemma}

\begin{proposition}\label{pr:localizationT1}
For any $p$ and, $t \in (0,1)$, small $\delta \geq 0$. Let $\varphi \in C_0^\infty(B_K)$, for any $L > 2$,
\[
 \vrac{\laps{\delta}(\scutA_{K+L} \laps{t} \varphi)}_{\frac{pn}{n+\delta p}} \aleq 2^{-L(\frac{n}{p'}+t)} \vrac{\laps{t} \varphi}_{p}.
\]
\[
 \vrac{\laps{t+\delta}(\scutA_{K+L} \lapms{t}	 \varphi)}_{\frac{pn}{n+\delta p}} \aleq 2^{-L\frac{n}{p'}} \vrac{\varphi}_{p}
\]
\end{proposition}
\begin{proof}
We prove only the first estimate, the second one follows by an analous argument.

For simplicity let us assume that $\vrac{\laps{t} \varphi}_{p} \leq 1$ and let $p_\delta := \frac{pn}{n+\delta p}$, $\rho := 2^K R$. By Lemma~\ref{la:local} and then Sobolev-Poincar\'e inequality
\[
 \vrac{\scutA_{(K+L)} \laps{t} \varphi}_{p_\delta} \aleq (2^L \rho )^{\frac{n}{p_\delta}}\ (2^L \rho)^{-n-t}\ \vrac{\varphi}_1 \aleq 2^{L(-\frac{n}{p'}-t+\delta)} \rho^{\delta}
\]
and by product rule and the same arguments as before,
\[
 \vrac{\nabla (\scutA_{(K+L)} \laps{t} \varphi)}_{\frac{pn}{n+\delta p}} \aleq  2^{L(-\frac{n}{p'}-t-(1-\delta))}\ \rho^{-(1-\delta)}
\]
Consequently, by interpolation we obtain the claim, multiplying the $0$-order exponents with $1-\delta$ and the $1$st-order exponent with $\delta$.
\end{proof}

\begin{proposition}\label{pr:classiclocalsobolev}
Let $s \in (0,n)$, $p \in (1,\frac{n}{s})$. Then for some $\sigma > 0$, for any $L \in \N$
\[
 \vrac{\lapms{s} f}_{\frac{np}{n-sp},B_\rho} \aleq \vrac{f}_{p,B_{2^L\rho}} + \sum_{l=1}^\infty 2^{-\sigma(L+l)}\ \vrac{f}_{p,B_{2^{L+l}\rho}}
\]
\end{proposition}
\begin{proof}
This follows from Lemma~\ref{la:local}, since
\[
 \vrac{\lapms{s} f}_{\frac{np}{n-sp},B_\rho} \aleq \vrac{\lapms{s} (\cut_{B_{2^L \rho}}f)}_{\frac{np}{n-sp}}  + \sum_{l=1}^\infty \vrac{\lapms{s} (\cutA_{B_{2^{L+l} \rho}\backslash B_{2^{L+l-1} \rho}}f)}_{\frac{np}{n-sp},B_\rho}.
\]
For the first term, we use Sobolev inequality, Theorem~\ref{th:sobolev:classic}, for the second term Lemma~\ref{la:local}.
\end{proof}

\begin{proposition}\label{pr:tripleintegral}
Let $s_1,s_2,s_3 \in [0,n)$ and $p_1,p_2,p_3 \in (1,\infty)$ so that
\[
 p^\ast_i := \frac{np_i}{n-s_i p_i} \in (1,\infty).
\]
If moreover
\[
 \sum_{i} \frac{1}{p_i} - \sum_{i} \frac{s_i}{n} = 1,
\]
then we have the following pseudo-local behaviour for any $L \in \N$:
\[
 \int_{\R^n} \lapms{s_1} (\cut_{B_\rho}f_1)\ \lapms{s_2} f_2\ \lapms{s_3} f_3 \aleq 
\]
\[
 \vrac{f_1}_{p_1,B_{2^L\rho}}\ \vrac{f_2}_{p_2,B_{2^L\rho}}\ \vrac{f_3}_{p_3,B_{2^L\rho}}
\]
\[
+ \sum_{l=1}^\infty 2^{-(L+t)\sigma}\ \vrac{f_1}_{p_1,B_{2^{L+l}\rho}}\ \vrac{f_2}_{p_2,B_{2^{L+l}\rho}}\ \vrac{f_3}_{p_3,B_{2^{L+l}\rho}}
\]
\end{proposition}
\begin{proof}
W.l.o.g. $L \geq 3$. We decompose
\begin{align*}
 &\int_{\R^n} \lapms{s_1} (\cut_{B_\rho}f_1)\ \lapms{s_2} f_2\ \lapms{s_3} f_3\\
 =& \sum_{i_1,i_2,i_3 = 0}^\infty \int_{\R^n} \chi_{i_1}\lapms{s_1} (\cut_{B_\rho}f_1)\ \lapms{s_2} (\chi_{i_2}f_2)\ \lapms{s_3} (\chi_{i_3}f_3)\\
\end{align*}
where
\[
 \chi_0 := \cut_{B_{2^L \rho}},
\]
and
\[
 \chi_i := \cutA_{B_{2^{L+l} \rho}\backslash B_{2^{L+l-1} \rho}}.
\]
The claim follows using repeatedely that by Lemma~\ref{la:local} for $j > i+1$
\[
 \vrac{\chi_{i} \lapms{s} (\cutA_{j} f)}_{\tilde{p}} \aleq 2^{-(L+j) \sigma} \vrac{\cut_{j} f}_{\tilde{q}} \leq 2^{-(L+j) \sigma} \vrac{f}_{\tilde{q},B_{2^{L+j}\rho}}
\]
We leave the details as an exercise.
\end{proof}

\section{Some Estimates with the Slobodeckij-Seminorm}
\begin{proposition}\label{pr:widmanguyest}
\[
 \int \limits_{B_{L}}\int \limits_{B_{L} \backslash B_{K}} \frac{|u(x)-u(y)|^{p_s}}{|x-y|^{n+sp_s}} dx\ dy \leq [u]_L^{p} - [u]_K^{p},
\]
\end{proposition}

\begin{proposition}\label{pr:fractionalint:scutK}
Let $\scut$ be from Definition~\ref{def:cutoffs}, $s \in (0,1)$, $p \in (1,\infty)$. Then
\[
 \int\limits_{\R^n} \frac{|\scut_K (x)-\scut_K(y)|^p}{|x-y|^{n+sp}}\ dx \aleq (2^K R)^{-sp} 
\]
\end{proposition}
\begin{proof}
Since $|\nabla \scut_K| \aleq  (2^K R)^{-1}$,
\[
 \int\limits_{\R^n} \frac{|\scut_K (x)-\scut_K(y)|^p}{|x-y|^{n+sp}}\ dx
 \]
 \[
 \aleq \int\limits_{|x-y|>2^KR} \frac{|\scut_K (x)-\scut_K(y)|^p}{|x-y|^{n+sp}}\ dx 
 +\int\limits_{|x-y|<2^KR} \frac{|\scut_K (x)-\scut_K(y)|^p}{|x-y|^{n+sp}}\ dx 
\]
 \[
 \aleq \int\limits_{|x-y|>2^KR} \frac{1}{|x-y|^{n+sp}}\ dx 
 +(2^K R)^{-p}\ \int\limits_{|x-y|<2^KR} \frac{1}{|x-y|^{n+(s-1)p}}\ dx.
\]
Now the claim follows from integration.
\end{proof}

\begin{proposition}\label{pr:etavsummvest}
Let $\scut$ be from Definition~\ref{def:cutoffs}, $s \in (0,1)$, $p \in (1,\infty)$. For any $L > K \in \N$
\[
 \int\limits_{B_L}\int\limits_{B_L} \frac{|\scut_K (x)-\scut_K(y)|^p\ |u(y)-(u)_K|^p}{|x-y|^{n+sp}}\ dx\ dy \aleq [u]_{B_{K+2},s,p}^p + ([u]_{B_{L},s,p}^p-[u]_{B_{K},s,p}^p).
\]
\end{proposition}
\begin{proof}
\begin{align*}
 &\int\limits_{B_L}\int\limits_{B_L} \frac{|\scut_K (x)-\scut_K(y)|^p\ |u(y)-(u)_K|^p}{|x-y|^{n+sp}}\ dx\ dy\\
 \leq&\int\limits_{B_{K+2}}\brac{\int\limits_{\R^n} \frac{|\scut_K (x)-\scut_K(y)|^p}{|x-y|^{n+sp}}\ dx}\ |u(y)-(u)_K|^p\ dy\\
 &+\int\limits_{B_L \backslash B_{K+2}}\brac{\int\limits_{B_L} \frac{|\scut_K (x)-\scut_K(y)|^p}{|x-y|^{n+sp}}\ dx}\ |u(y)-(u)_K|^p\ dy.
\end{align*}
The first term is estimated by Proposition~\ref{pr:fractionalint:scutK} against $[u]_{B_K,s,p}^p$. For the second term observe that $\scut_K(x)-\scut_K(y) = 0$ if both $x,y \in B_L \backslash B_{K+1}$, so it becomes
\begin{align*}
&\int\limits_{B_L \backslash B_{K+2}}\brac{\int\limits_{B_{K+1}} \frac{|\scut_K (x)-\scut_K(y)|^p}{|x-y|^{n+sp}}\ dx}\ |u(y)-(u)_K|^p\ dy\\
\aleq &\sum_{l=K+3}^{L} (2^{l} R)^{-n-sp} (2^K R)^{n} \int\limits\cutA_l(y) |u(y)-(u)_K|^p\ dy\\
\aleq &\sum_{l=K+3}^{L} (2^{l} R)^{-n-sp} \int \cutA_l(y) \int \limits_{B_K} |u(y)-u(z)|^p\ dz\ dy\\
\aleq &\sum_{l=K+3}^{L} \int\cutA_l(y) \int \limits_{B_K} \frac{|u(y)-u(z)|^p}{|z-y|^{n+sp}}\ dz\ dy\\
\leq &\int\limits_{B_L \backslash B_K} \int \limits_{B_K} \frac{|u(y)-u(z)|^p}{|z-y|^{n+sp}}\ dz\ dy.
\end{align*}
Then we use Proposition~\ref{pr:widmanguyest} to conclude.
\end{proof}

\begin{proposition}\label{pr:psiest}
Let 
\[
 \psi(x) := \scut_K(x) (u(x)-(u)_{K}),
\]
then
\[
 [\psi]_{s,p,\R^n} \aleq [u]_{K+1}.
\]
\end{proposition}
\begin{proof}
First of all, we have
\[
 [\psi]_{s,p,\R^n} \aleq \int\limits_{\R^n \backslash B_{K+2}} \int\limits_{B_{K+1}} \frac{|\psi(y)|^p}{|x-y|^{n+sp}}\ dy\ dx + [\psi]_{s,p,B_{K+1}}^p
\]
We integrate the first term in $x$, observing $|x-y| > 2^K R$. Since moreover
\[
 |\psi(x)-\psi(y)| \leq |\scut_0(x)-\scut_0(y)|\ |u(y)-(u)_0| + |u(x)-u(y)|,
\]
we can further estimate the second term and using Proposition~\ref{pr:etavsummvest} arrive at
\[
 [\psi]_{s,p,\R^n} \aleq (2^KR)^{-sp}\int\limits_{B_{K+1}} |\psi(y)|^p\ dy\ dx + [u]_{s,p,B_{K+1}}^p
\]
Finally we use Jensen's inequality and have
\[
 \int\limits_{B_{K+1}} |\psi(y)|^p dy\ \aleq (2^K R)^{-n} \int\limits_{B_{K+1}} \int\limits_{B_K} |u(y)-u(z)|^p \aleq (2^K R)^{sp} [u]_{s,p,B_{K+1}}.
\]

\end{proof}

\begin{proposition}\label{pr:lapmsTest}
For all small $\delta > 0$ 
\[
  \vrac{T_{B_{\rho},s+\delta} u^i(z)}_{\frac{n}{n-s-\delta}} \aleq [u]_{B_\rho}^{p-1}.
\]
\end{proposition}
\begin{proof}
Pick $f \in \Sw(\R^n)$,  $\vrac{f}_{\frac{n}{s+\delta}} \leq 1$ such that
\begin{align*}
  \vrac{T_{B_{\rho},s+\delta} u^i(z)}_{\frac{n}{n-s-\delta}} \aleq &\int \limits_{\R^n} T_{B_{\rho},s+\delta} u^i(z)\ f\\
  \aleq &\int \limits_{B_\rho}\int \limits_{B_\rho} \frac{|u(x)-u(y)|^{p-1} |\lapms{s+\delta} f(x) -\lapms{s+\delta} f(y)|}  {|x-y|^{n+sp}}\\
  \aleq &[u]_{B_\rho}^{p_s-1}\ \brac{\int \limits_{B_\rho}\int \limits_{B_\rho} \frac{|\lapms{s+\delta} f(x) -\lapms{s+\delta} f(y)|^p}  {|x-y|^{n+sp}}\ dx\ dy}^{\frac{1}{p}}\\
  \aleq &[u]_{B_\rho}^{p_s-1}\ \vrac{f}_{\frac{n}{s+\delta}}.
\end{align*}
The last estimate comes from Lemma~\ref{la:slobsob}.
\end{proof}

\begin{proposition}\label{pr:uxmuypm1varphi}
Fix a scale $B_\rho(x_0)$. Let $\varphi \in C_0^\infty(B_\rho)$. Then for any $L \geq 2$
\[
 \int \limits_{\R^n \backslash B_L} \int \limits_{\R^n} \frac{|u(x)-u(y)|^{p-1} |\varphi(x)|}{|x-y|^{n+sp}}\ dx\ dy \aleq [\varphi]_{s,p,B_1}\ \sum_{l=1}^\infty 2^{-\sigma (L+l)} [u]_{s,p,B_{L+l+1}}^{p-1}
\]
\end{proposition}
\begin{proof}
We may assume that $\rho = 1$.
\begin{align*}
 &\int \limits_{B_{L+l} \backslash B_{L+l-1}} \int \limits_{B_0} \frac{|u(x)-u(y)|^{p-1} |\varphi(x)|}{|x-y|^{n+sp}}\ dx\ dy\\
 \aeq & 2^{-(L+l)(n+sp)} \int \limits_{B_{L+l} \backslash B_{L+l-1}} \int \limits_{B_0} |u(x)-u(y)|^{p-1} |\varphi(x)| dx\ dy\\
 \aleq & 2^{-(L+l)(sp)} 	\int \limits_{B_0} |u(x)-(u)_{L+l}|^{p-1} |\varphi(x)| dx\ dy\\
 & + 2^{-(L+l)(n+sp)} \int \limits_{B_{L+l} \backslash B_{L+l-1}} |(u)_{L+l}-u(y)|^{p-1} \int \limits_{B_0} |\varphi(x)| dx\ dy\\
 \aleq & (2^{-(L+l)s} + 2^{-(L+l)(\frac{n}{p'}+s)})\ [u]_{s,p,B_{L+l+1}}^{p-1}\ [\varphi]_{s,p,B_1}.
\end{align*}

\end{proof}

\end{appendix}

\bibliographystyle{plain}%
\bibliography{bib}%

\end{document}